\documentclass[12pt]{article}
\usepackage[utf8]{inputenc}
\usepackage{amsmath}
\usepackage{amssymb}
\usepackage[numbers]{natbib}
\usepackage[unicode=true,
 bookmarks=false,
 breaklinks=false,pdfborder={0 0 1},backref=section,colorlinks=false]
 {hyperref}

\makeatletter
\usepackage{graphics}
\usepackage{url}

\newenvironment {proof}{{\noindent\bf Proof.}}{\hfill $\Box$ \medskip}

\newtheorem{theorem}{Theorem}[section]
\newtheorem{lemma}[theorem]{Lemma}

\newtheorem{proposition}[theorem]{Proposition}
\newtheorem{remark}[theorem]{Remark}

\renewcommand {\theequation}{\arabic{section}.\arabic{equation}}

\def \hat{\widehat}
\def \tilde{\widetilde}
\def \bar{\overline}

\allowdisplaybreaks

\makeatother

\begin{document}

\title{Particle representation for the solution of the filtering problem.
Application to the error expansion of filtering discretizations\thanks{Paper submitted to the Festschrift in honor of Hiroshi Kunita as part
of a memorial volume in the Journal of Stochastic Analysis.}}
\author{D. Crisan\thanks{Department of Mathematics, Imperial College London, Huxley's Building,180
Queen's Gate, London SW7 2AZ, United Kingdom, E-mail: dcrisan@imperial.ac.uk }, T. G. Kurtz\thanks{Departments of Mathematics and Statistics, University of Wisconsin-
Madison, 48 Lincoln Drive, Madison WI 53706-1388, USA, E-mail: kurtz@math.wisc.edu } and S. Ortiz-Latorre\thanks{Department of Mathematics, University of Oslo, P.O. Box 1053 Blindern,
N-0316 Oslo, Norway, E-mail: salvadoo@math.uio.no}}
\date{\today}
\maketitle
\begin{abstract}
We introduce a weighted particle representation for the solution of
the filtering problem based on a suitably chosen variation of the
classical de Finetti theorem. This representation has important theoretical
and numerical applications. In this paper, we explore some of its
theoretical consequences. The first is to deduce the equations satisfied
by the solution of the filtering problem in three different frameworks:
the signal independent Brownian measurement noise model, the spatial
observations with additive white noise model and the cluster detection
model in spatial point processes. Secondly we use the representation
to show that a suitably chosen filtering discretisation converges
to the filtering solution. Thirdly we study the leading error coefficient
for the discretisation. We show that it satisfies a stochastic partial
differential equation by exploiting the weighted particle representation
for both the approximation and the limiting filtering solution.


\end{abstract}

\section{Introduction}

\label{sectintro} Many phenomena of interest are not completely observable,
so it is natural to look for ways of estimating what is not observable
about a phenomenon in terms of what is. A natural approach to this
problem is to create a mathematical model for the phenomenon that
relates what is observable to what is not. Then the model can be used
to constrain or estimate the possibilities for the unobserved quantities
in terms of the observed quantities. If the mathematical model is
stochastic, then a natural way of formulating the solution of this
problem is to compute the conditional distribution of what is not
known given what is known. If the unknown and known quantities are
evolving in time, the problem of computing these conditional distributions
as functions of time is referred to as {\em stochastic filtering}.

Stochastic filtering has a illustrious history that can be traced
back to the work of Kolmogorov, Krein and Wiener from the 1940's\footnote{The interested reader can consult \citep{bc2009} and \citep{cr2011}
for historical accounts of the subject.}. Over the last fifty years, Professor Kunita has made fundamental
contributions to the study of the stochastic filtering problem. A
representative selection of his work on the subject is incorporated
in \citep{FKK72,k2011,k1971,Kun79,k1981,k1982,K1991,KH1991}, and
covers : the stochastic partial differential equations connected with
stochastic filtering for the classical model considered here in Section
\ref{mod1}, the stochastic filtering problem for models in which
the signal is a general semi-martingale, the ergodic properties of
nonlinear filtering processes, the associated stability and approximation
problems in nonlinear filtering theory, the nonlinear filtering Cauchy
problem, the asymptotic behavior of the nonlinear filtering errors
of {M}arkov processes, the analysis of the innovation process, the
long time behavior of the solution of the filtering equations, etc.

The work \citep{FKK72} deserves special consideration. It contains
a self-contained, fully rigorous derivation of filtering equations.
It is based on an approach that requires the innovation process first
considered by Kailath only a year earlier. It also uses the existence
of a reference probability measure obtained from the original one
by means of a transformation due to Girsanov that was, at the time,
barely a decade old. It used a (by now classical) representation of
square integrable martingales (appearing in Kunita's earlier work
with Watanabe) to remove the assumption of independence between the
measurement noise and the signal. The treatment in \citep{FKK72}
of the filtering problem was much cleaner than existing contemporary
works and allowed, among other things, for the treatment of controlled
system processes.

The current work offers an alternative to Kunita's treatment of the
filtering equations. In common with \citep{FKK72}, we still make
use of the reference probability measure (as described below). However,
the main tool for the derivation of filtering equations is a certain
weighted particle representation for the solution of the filtering
problem. Let us describe next the intuition behind this representation:

The simplest version of the filtering problem is one in which the
model consists of two random variables, say $X$ and $Y$, where $Y$
is known to an observer and $X$ is not. Assume the random variables
are defined on a sample space $(\Omega,{\cal F},P)$, with $X$ taking
values in a space ${\cal U}$ and $Y$ taking values in a space ${\cal O}$,
which we will always take to be complete, separable metric spaces.
In all the examples we consider, these will be function spaces. Typically,
we characterize the conditional distribution in terms of the conditional
expectations $\mathbb{E}^{P}\left[f(X)|Y\right]$ for a sufficiently
large class of functions $f$.

Central to our analysis is the notion of a {\em reference probability
measure}. If $P<<Q$ with $dP=LdQ$, Bayes formula says 
\[
\mathbb{E}^{P}\left[f(X)|Y\right]=\frac{\mathbb{E}^{Q}\left[f(X)L|Y\right]}{\mathbb{E}^{Q}\left[L|Y\right]}.
\]
If $X=h(U,Y)$, $L=L(U,Y)$, and $U$ and $Y$ are independent under
$Q$, then 
\[
\mathbb{E}^{P}\left[f(X)|Y\right]=\frac{\int f(h(u,Y))L(u,Y)\mu_{U}(du)}{\int L(u,Y)\mu_{U}(du)},
\]
where $\mu_{U}$ is the distribution of $U$. If we can find such
a $Q$, that will be our reference probability measure.

These comments suggest a method for simplifying the calculation of
conditional distributions: Find a reference probability measure under
which what we don't know is independent of what we do know. Then,
as in \citep{HM69}, let $U_{1},U_{2},\ldots$ be iid with distribution
$\mu_{U}$, and we have 
\[
\mathbb{E}^{P}\left[f(X)|Y\right]=\lim_{N\rightarrow\infty}\frac{\sum_{k=1}^{N}f(h(U_{k},Y))L(U_{k},Y)}{\sum_{k=1}^{N}L(U_{k},Y)}.
\]
Note that $(U_{1},Y),(U_{2},Y),\ldots$ is an exchangeable sequence
with tail $\sigma$-algebra ${\cal T}=\sigma(Y)$ (see Corollary 7.25
of \citep{Kal02}), and de Finetti's theorem gives 
\begin{eqnarray*}
\lim_{N\rightarrow\infty}\frac{1}{N}\sum_{k=1}^{N}f(h(U_{k},Y))L(U_{k},Y) & = & \mathbb{E}^{Q}\left[f(U_{1},Y)L(U_{1},Y)|{\cal T}\right]\\
 & = & \mathbb{E}^{Q}\left[f(U_{1},Y)L(U_{1},Y)|Y\right]\\
 & = & \int f(h(u,Y))L(u,Y)\mu_{U}(dx)\qquad a.s.\;Q.
\end{eqnarray*}

In the context of stochastic processes and the Kallianpur-Striebel
formula (\citep{KS68}), this limit suggests a natural approach to
the derivation and representation of filtering equations. We introduce
a weighted particle representation for the solution of the filtering
problem based on a suitably chosen variation of the classical de Finetti
theorem. This representation has important theoretical and numerical
applications. In this paper, we explore some of its theoretical consequences.

The first is to deduce the equations satisfied by the solution of
the filtering problem in three different frameworks: the signal independent
Brownian measurement noise model, the spatial observations with additive
white noise model, and the cluster detection model in spatial point
processes. We cover this topic in Section \ref{sectdrv}.

Second, we use the representation to show that a suitably chosen filtering
discretization converges to the filtering solution. We cover this
topic in Section \ref{subsec: diffusion in GWN}. This discretization
is one of three procedures required to develop any numerical method
approximating the solution of the filtering problem. See, for example,
Chapters 8, 9 and 10 in \citep{bc2009} for concrete examples of numerical
schemes for solving the filtering problem.

Third, we study the leading error coefficient for the discretization
introduced in Section \ref{subsec: diffusion in GWN}. In Section
\ref{secterr}, we show that it satisfies a stochastic partial differential
equation by exploiting the weighted particle representation for both
the approximation and the limiting filtering solution. Based on these
representations, an extension of the classical Richardson extrapolation
result can also be obtained. This is the subject of a subsequent work.

Particle representations are flexible tools that can be used for many
other stochastic dynamical systems. In this paper, particle representations
are used to characterize the solution of the filtering problem by
deducing the corresponding filtering equation, and it is also used
to show the convergence of a certain discretization of the filtering
solution. However, particle representations have many other applications.
In \citep{KX99}, they are used to prove uniqueness for a class of
stochastic partial differential equations that includes filtering
equations. In \citep{CKL2014}, particle representations are used
to study the solution of a nonlinear stochastic partial differential
equation. In particular, the authors show, under mild nondegeneracy
conditions on the coefficients, that the solution charges every open
set and, under slightly stronger conditions, that the solution is
absolutely continuous with respect to Lebesgue measure with strictly
positive density. Such results would be very hard to obtain (under
the same general assumptions) by other methods such as PDE methods
(Sobolev embedding theorems) or Malliavin calculus. Separately, in
\citep{CJK2018}, a similar particle representation is used to study
a class of semilinear stochastic partial differential equations with
Dirichlet boundary conditions that includes the stochastic Allen-Cahn
equation and the $\Phi_{d}^{4}$ equation of Euclidean quantum field
theory. Particle representations arise naturally in the study of McKean-Vlasov
type models, for example, \citep{KP96,KK10,CKL2014} where the representations
are used to prove limit theorems.

We should emphasize that what we are deriving here are {\em particle
representations\/} of the filter rather than {\em particle approximations}.\footnote{The type of weighted particle representations considered here were
mentioned briefly in \citep{KP96}} There is a massive area of research regarding {\em particle approximations}
of the distributions of evolving dynamical system, which we shall
not discuss here.

\section{Derivation of filtering equations}

\label{sectdrv} In the Introduction, we applied de Finetti's theorem
to derive a representation of a conditional expectation in terms of
what we called a reference probability measure. In this section, we
use this argument to derive stochastic equations giving the solution
of the filtering problem in three different settings. The first of
these is the familiar observation of a diffusion in Gaussian white
noise. The second is similar, but includes a noise process that is
common to both the signal and the observation. In addition, the observation
process is infinite dimensional. In the third example, the signal
and observations are given by spatial point processes.

To avoid certain technicalities, we assume that all $\sigma$-algebras
are complete and all filtrations are complete and right continuous.

\subsection{Observation of a diffusion in Gaussian white noise}

\label{mod1}

\subsubsection{The model}

The signal is given by an Itô equation in $\mathbb{R}^{d_{X}}$, 
\begin{equation}
X(t)=X(0)+\int_{0}^{t}\sigma(X(s))dB(s)+\int_{0}^{t}b(X(s))ds,\label{sig1}
\end{equation}
for $d_{B}$-dimensional standard Brownian motion $B$, continuous
$d_{X}\times d_{B}$ matrix-valued $\sigma$, and continuous $\mathbb{R}^{d_{X}}$-valued
$b$, and the observation by 
\begin{equation}
Y(t)=\int_{0}^{t}h(X(s))ds+W(t),\label{obs1}
\end{equation}
where $h:\mathbb{R}^{d_{X}}\rightarrow\mathbb{R}^{d_{Y}}$ is measurable
and $W$ is a $\mathbb{R}^{d_{Y}}$-valued standard Brownian motion
that is independent of $B$. What is known to the observer is $Y$
and what is not known is $X$, or assuming uniqueness for (\ref{sig1}),
$B$. We assume that $X$ does not explode, as would be the case if
$\sigma$ and $b$ have at most linear growth, that is 
\begin{equation}
\left|\sigma(x)\right|+\left|b(x)\right|\leq K_{1}+K_{2}\left|x\right|.\label{lingr}
\end{equation}
For simplicity, we assume $\mathbb{E}[\left|X(0)\right|^{m}]<\infty$
for all $m>0$, and note that under the linear growth assumption,
an exercise with Itô's formula shows that for each $m\geq2$, implies
\begin{equation}
\mathbb{E}[\sup_{0\leq s\leq t}\left|X(s)\right|^{m}]<D_{1}^{m}e^{D_{2}^{m}t},\quad\forall t>0,\label{momest1}
\end{equation}
for appropriate constants $D_{1},D_{2}$, see Proposition 7.2 in \citep{Pa18}.

\subsubsection{The reference probability space}

We take $(\Omega,{\cal F},Q)$ to be a probability space on which
are defined independent Brownian motions $B$ and $Y$, both independent
of $X(0)$, with the same dimensions as $B$ and $Y$ above, that
is, under $Q$, what is known is independent of what is not known.
We note that many presentations of filtering problems begin with $(\Omega,{\cal F},P)$
and obtain $Q$ by change of measure from $P$. That approach requires
$P$ and $Q$ to be equivalent in the sense that $P<<Q$ and $Q<<P$.
In many settings, constructing the model starting with $Q$ is more
straightforward and does not require $Q<<P$.

Assume that $B$ and $W$ are $\{{\cal F}_{t}\}$-Brownian motions
for a filtration $\{{\cal F}_{t}\}$ (which, as noted above, we assume
is complete and right continuous) and that the signal $X$ is defined
on $(\Omega,{\cal F},Q)$ as the solution of (\ref{sig1}). Let 
\[
L(t)=\exp\left\{ \int_{0}^{t}h^{T}(X(s))dY(s)-\frac{1}{2}\int_{0}^{t}|h|^{2}(X(s))ds\right\} ,
\]
that is, 
\[
L(t)=1+\int_{0}^{t}L(t)h^{T}(X(s))dY(s),
\]
and assume that $h$ satisfies conditions ensuring that $L$ is a
martingale. For example, we can assume $h$ is bounded. Then defining
$dP_{|{\cal F}_{t}}=L(t)dQ_{|{\cal F}_{t}}$, under $P$, by Theorem
\ref{cmmart2}, $B$ and $W$ given by 
\[
W(t)=Y(t)-\int_{0}^{t}h(X(s))ds
\]
are independent standard Brownian motions. Consequently, under $P$,
$X$ and $Y$ have the joint distributions of (\ref{sig1}) and (\ref{obs1}).

\subsubsection{Filtering equations}

Let $\{{\cal F}_{t}^{Y}\}$ be the (completed) filtration generated
by the observations $Y$. Then, assuming $\mathbb{E}^{Q}[\left|\varphi(X(t))L(t)\right|]<\infty$,
we have the Kallianpur-Striebel formula \citep{KS68} 
\[
\mathbb{E}^{P}\left[\varphi(X(t))|{\cal F}_{t}^{Y}\right]=\frac{\mathbb{E}^{Q}\left[\varphi(X(t))L(t)|{\cal F}_{t}^{Y}\right]}{\mathbb{E}^{Q}\left[L(t)|{\cal F}_{t}^{Y}\right]}.
\]
Let $X_{1},X_{2},\ldots$ be iid copies of $X$ that are independent
and independent of $Y$ under $Q$, and let 
\[
L_{k}(t)=1+\int_{0}^{t}L_{k}(s)h(X_{k}(s))^{T}dY(s).
\]
We define the {\em unnormalized conditional distribution} $\rho$
by 
\[
\rho_{s}(\varphi)\equiv\mathbb{E}^{Q}\left[\varphi(X(s))L(s)|{\cal F}_{s}^{Y}\right]=\mathbb{E}^{Q}\left[\varphi(X_{k}(s))L_{k}(s)|{\cal F}_{s}^{Y}\right],
\]
and the exchangeability of $\{(X_{k},Y)\}$ ensures 
\begin{equation}
\lim_{N\rightarrow\infty}\frac{1}{N}\sum_{k=1}^{N}\varphi(X_{k}(s))L_{k}(s)=\mathbb{E}^{Q}[\varphi(X(s))L(s)|{\cal F}_{s}^{Y}].\label{mc1}
\end{equation}
With reference to Appendix \ref{limpr}, since $\{(X_{k},L_{k})\}$
is exchangeable, by de Finetti's theorem, the sequence determines
a random probability measure, which we will call the de Finetti measure,
\[
\Xi=\lim_{N\rightarrow\infty}\frac{1}{N}\sum_{k=1}^{N}\delta_{(X_{k},L_{k})}\quad a.s.
\]
on $C_{\mathbb{R}^{d_{X}}\times[0,\infty)}[0,\infty)$ and a probability
measure-valued process 
\[
V(\cdot)=\lim_{N\rightarrow\infty}\frac{1}{N}\sum_{k=1}^{N}\delta_{(X_{k}(\cdot),L_{k}(\cdot))}\quad\mbox{{\rm in probability}},
\]
in $C_{{\cal P}(\mathbb{R}^{d_{X}}\times[0,\infty))}[0,\infty)$.
Of course, the unnormalized conditional distribution is 
\[
\rho_{t}(C)=\int_{\mathbb{R}^{d_{X}}\times[0,\infty)}a{\bf 1}_{C}(x)V(dx\times da,t).
\]
\begin{lemma} Assume $|h|$ is bounded. Then as with (\ref{momest1}),
for each $m\geq2$, there exists $D_{3}$ such that 
\begin{equation}
\mathbb{E}^{Q}[\sup_{0\leq s\leq t}L(s)^{m}]\leq e^{D_{3}t}.\label{momest2}
\end{equation}
\end{lemma}

We assume (\ref{momest1}) and (\ref{momest2}) for all $m>0$. Then,
for $\varphi\in C^{2}(\mathbb{R}^{d_{X}})$ satisfying $|\varphi(x)|\leq C_{1}+C_{2}|x|^{m}$,
\[
\varphi(X(t))=\varphi(X(0))+\int_{0}^{t}\nabla\varphi^{T}(X(s))\sigma(X(s))dB(s)+\int_{0}^{t}A\varphi(X(s))ds,
\]
where for $a(x)=\sigma(x)\sigma^{T}(x)$ 
\[
A\varphi(x)=\frac{1}{2}\sum_{i,j}a_{ij}(x)\partial_{i}\partial_{j}\varphi(x)+b(x)\cdot\nabla\varphi(x),
\]
and 
\begin{eqnarray*}
\varphi(X(t))L(t) & = & \varphi(X(0))+\int_{0}^{t}L(s)d\varphi(X(s))ds+\int_{0}^{t}\varphi(X(s))dL(s)\\
 & = & \varphi(X(0))+\int_{0}^{t}L(s)\nabla\varphi(X(s))^{T}\sigma(X(s))dB(s)\\
 &  & \quad+\int_{0}^{t}L(s)A\varphi(X(s))ds+\int_{0}^{t}\varphi(X(s))L(s)h(X(s))^{T}dY(s).
\end{eqnarray*}
Then for $\{(X_{k},L_{k})\}$ given above, 
\begin{eqnarray}
\varphi(X_{k}(t))L_{k}(t) & = & \varphi(X_{k}(0))+\int_{0}^{t}L_{k}(s)\nabla\varphi(X_{k}(s))^{T}\sigma(X_{k}(s))dB_{k}(s)\label{geq1}\\
 &  & \qquad+\int_{0}^{t}L_{k}(s)A\varphi(X_{k}(s))ds\nonumber \\
 &  & \qquad+\int_{0}^{t}\varphi(X_{k}(s))L_{k}(s)h(X_{k}(s))^{T}dY(s).\nonumber 
\end{eqnarray}
We claim that we can average both sides as in (\ref{mc1}) and obtain
the following:

\begin{theorem} For the model in Section \ref{mod1}, the unnormalized
conditional distribution $\rho$ satisfies 
\begin{equation}
\rho_{t}(\varphi)=\rho_{0}(\varphi)+\int_{0}^{t}\rho_{s}(A\varphi)ds+\int_{0}^{t}\rho_{s}(\varphi h)dY(s),\label{mod1zak}
\end{equation}
the Zakai equation, and by Itô's formula, we have the Kushner-Stratonovich
equation. 
\begin{eqnarray*}
\pi_{t}\varphi & = & \mathbb{E}^{P}[\varphi(X(t))|{\cal F}_{t}^{Y}]=\frac{\rho_{t}(\varphi)}{\rho_{t}(1)}\\
 & = & \frac{\rho_{0}(\varphi)}{\rho_{0}(1)}+\int_{0}^{t}\frac{1}{\rho_{s}(1)}d\rho_{s}(\varphi)-\int_{0}^{t}\frac{\rho_{s}(\varphi)}{\rho_{s}(1)^{2}}d\rho_{s}(1)\\
 &  & \qquad+\int_{0}^{t}\frac{\rho_{s}(\varphi)}{\rho_{s}(1)^{3}}d[\rho_{\cdot}(1)]_{s}-\int_{0}^{t}\frac{1}{\rho_{s}(1)^{2}}d[\rho_{\cdot}(\varphi),\rho_{\cdot}(1)]_{s}\\
 & = & \pi_{0}\varphi+\int_{0}^{t}\pi_{s}A\varphi ds+\int_{0}^{t}(\pi_{s}\varphi h-\pi_{s}\varphi\pi_{s}h)dY(s)\\
 &  & \qquad+\int_{0}^{t}\sigma^{2}\pi_{s}\varphi\pi_{s}h^{2}ds-\int_{0}^{t}\sigma^{2}\pi_{s}\varphi h\pi_{s}hds\\
 & = & \pi_{0}\varphi+\int_{0}^{t}\pi_{s}A\varphi ds+\int_{0}^{t}(\pi_{s}\varphi h-\pi_{s}\varphi\pi_{s}h)(dY(s)-\pi_{s}hds).
\end{eqnarray*}
\end{theorem}

\begin{proof} The term on the left and the first term on the right
of (\ref{geq1}) average as in (\ref{mc1}). The average over $1\leq k\leq N$
of the second term on the right is a continuous, mean zero martingale
with quadratic variation 
\[
\frac{1}{N^{2}}\sum_{k=1}^{N}\int_{0}^{t}L_{k}(s)^{2}\nabla\varphi(X_{k}(s))^{T}\sigma(X_{k}(s))\sigma(X_{k}(s))^{T}\nabla\varphi(X_{k}(s))ds,
\]
which, under the growth and moment conditions above, converges to
zero implying the average converges to zero by Doob's inequality.
The averages of the integrands in the last two terms converge to the
integrands in the last two terms of (\ref{mod1zak}) by Lemma \ref{exchconv},
so the next to the last terms converges by elementary calculus and
the last term converges by the stochastic integral convergence results
in \citep{KP91}. \end{proof}

In the terminology of \citep{KX99}, the infinite sequence $\{(X_{k},L_{k})\}$
gives a {\em particle representation\/} of the Zakai equation.
The point here is not just that $\{(X_{k},L_{k})\}$ gives a derivation
of the Zakai equation. The representation can, for example, be used
to prove uniqueness (see \citep{KX99}) and derive approximations
(see Section \ref{secterr}).

\subsection{Spatial observations with additive white noise}

\label{mod2}

\subsubsection{The model}

The basic outline of the argument above works in many different situations.
We again take the signal to be a diffusion in $\mathbb{R}^{d_{X}}$,
but now we assume that the stochastic inputs include both a $d_{B}$-dimensional
standard Brownian motion $B$ and a space-time Gaussian white noise
$W$. In particular, 
\begin{equation}
X(t)=X(0)+\int_{0}^{t}\sigma(X(s))dB(s)+\int_{0}^{t}b(X(s))ds+\int_{S_{0}\times[0,t]}\alpha(X(s),u)W(du\times ds),\label{sig2}
\end{equation}
where $\mathbb{E}[W(C,t)]=0$ and 
\[
\mathbb{E}[W(C,t)W(D,s)]=\mu_{0}(C\cap D)t\wedge s,
\]
$t,s\geq0$ and $C,D\in{\cal B}(S_{0})$, the Borel sets for some
complete, separable metric space, $S_{0}$.

We assume that the observations are given by $Y(C,t)$, $t\geq0$
and $C\in{\cal B}(S_{0})$, where 
\begin{equation}
Y(C,t)=\int_{0}^{t}\int_{C}h(X(s),u)\mu_{0}(du)ds+W(C,t).\label{obs2}
\end{equation}
Consequently, $\sigma$ is a $d_{X}\times d_{B}$-dimensional matrix-valued
function, $b$ and $\alpha$ are $\mathbb{R}^{d_{X}}$-valued, and
$h$ is $\mathbb{R}$-valued. For simplicity, assume $\sigma,b,\alpha,$
and $h$ are bounded and continuous and that $\mu_{0}$ is a finite
measure.

Then the generator for $X$ is 
\[
A\varphi(x)=\frac{1}{2}\sum a_{ij}(x)\partial_{i}\partial_{j}\varphi(x)+\sum b_{i}(x)\partial_{i}\varphi(x),\quad\varphi\in C_{c}^{2}\left(\mathbb{R}^{d_{X}}\right),
\]
where 
\[
a(x)=\sigma(x)\sigma(x)^{T}+\int_{S_{0}}\alpha(x,u)\alpha(x,u)^{T}\mu_{0}(du).
\]

We can write 
\begin{eqnarray}
X(t) & = & X(0)+\int_{0}^{t}\sigma(X(s))dB(s)+\int_{S_{0}\times[0,t]}\alpha(X(s),u)Y(du\times ds)\label{sig2b}\\
 &  & \quad+\int_{0}^{t}(b(X(s))-\int_{S_{0}}\alpha(X(s),u)h(X(s),u)\mu_{0}(du))ds,\nonumber 
\end{eqnarray}
so $Y$ is what we know and $B$ is what we don't know.

\subsubsection{The reference probability space}

Consequently, we assume $B$ and $Y$ are defined on a measurable
space $(\Omega,{\cal F})$, and there is a probability distribution
$Q$ on ${\cal F}$ such that under $Q$, $Y$ is Gaussian white noise
on $S_{0}\times[0,\infty)$ with $\mathbb{E}[Y(C,t)]=0$, $t\geq0$,
$C\in{\cal B}(S_{0})$ and 
\[
\mathbb{E}[Y(C,t)Y(D,s)]=\mu_{0}(C\cap D)t\wedge s,
\]
and $B$ is a standard $d_{B}$ dimensional Brownian motion independent
of $Y$. Both are independent of $X(0)$.

Then take $dP_{|{\cal F}_{t}}=L(t)dQ_{|{\cal F}_{t}}$ where 
\[
L(t)=1+\int_{S_{0}\times[0,t]}L(s)h(X(s),u)Y(du\times ds),
\]
and under $P$, $(X,Y)$ has the joint distribution of the original
model.

\subsubsection{Filtering equations}

Under $Q$, $X$ is a diffusion with generator 
\[
A^{Q}\varphi(x)=\frac{1}{2}\sum a_{ij}(x)\partial_{i}\partial_{j}\varphi(x)+\sum c_{i}(x)\partial_{i}\varphi(x),
\]
where 
\[
c_{i}(x)=b_{i}(x)-\int_{S_{0}}\alpha(x,u)h(x,u)\mu_{0}(du),
\]
and under $P$, $X$ is a diffusion with the original generator 
\[
A\varphi(x)=\frac{1}{2}\sum a_{ij}(x)\partial_{i}\partial_{j}\varphi(x)+\sum b_{i}(x)\partial_{i}\varphi(x),
\]
that is, $X$ is the signal of the original model.

Then 
\begin{eqnarray*}
 &  & \varphi(X(t))L(t)\\
 &  & \quad=\varphi(X(0))+\int_{0}^{t}L(s)\nabla\varphi(X(s))^{T}\sigma(X(s))dB(s)\\
 &  & \qquad+\int_{S_{0}\times[0,t]}L(s)(\nabla\varphi(X(s))\cdot\alpha(X(s),u))Y(du\times ds)\\
 &  & \qquad+\int_{0}^{t}L(s)A^{Q}\varphi(X(s))ds+\int_{S_{0}\times[0,t]}L(s)\varphi(X(s))h(X(s),u))Y(du\times ds)\\
 &  & \qquad+\int_{0}^{t}L(s)\nabla\varphi(X(s))^{T}\int_{S_{0}}\alpha(X(s),u)h(X(s),u)\mu_{0}(du)ds.
\end{eqnarray*}

To obtain the particle representation, we let $B_{k}$ be independent,
standard Brownian motions, independent of $Y$ on $(\Omega,{\cal F},Q)$.
Let 
\begin{eqnarray*}
X_{k}(t) & = & X_{k}(0)+\int_{0}^{t}\sigma(X_{k}(s))dB_{k}(s)+\int_{S_{0}\times[0,t]}\alpha(X_{k}(s),u)Y(du\times ds)\\
 &  & \quad+\int_{0}^{t}\left(b(X_{k}(s))-\int_{S_{0}}\alpha(X_{k}(s),u)h(X_{k}(s),u)\mu_{0}(du\right)ds\\
 &  & \qquad\\
L_{k}(t) & = & 1+\int_{S_{0}\times[0,t]}L_{k}(s)h(X_{k}(s),u)Y(du\times ds).
\end{eqnarray*}
Then, as before, 
\[
\rho_{t}(\varphi)=\mathbb{E}^{P}[\varphi(X(t))L(t)|{\cal F}_{t}^{Y}]=\lim_{N\rightarrow\infty}\frac{1}{N}\sum_{k=1}^{N}\varphi(X_{k}(t))L_{k}(t).
\]

Since 
\begin{eqnarray*}
 &  & \varphi(X_{k}(t))L_{k}(t)\\
 &  & \quad=\varphi(X_{k}(0))+\int_{0}^{t}L_{k}(s)\nabla\varphi(X_{k}(s))^{T}\sigma(X_{k}(s))dB_{k}(s)\\
 &  & \qquad+\int_{S_{0}\times[0,t]}L_{k}(s)(\nabla\varphi(X_{k}(s))\cdot\alpha(X_{k}(s),u)Y(du\times ds)\\
 &  & \qquad+\int_{0}^{t}L_{k}(s)A\varphi(X_{k}(s))ds+\varphi(X_{k}(s))h(X_{k}(s),u))Y(du\times ds),
\end{eqnarray*}
under appropriate moment conditions and applying convergence results
of \citep{KP96} to the $Y$ integral instead of \citep{KP91}, averaging
gives

\begin{theorem} 
\begin{eqnarray*}
\rho_{t}(\varphi) & = & \rho_{0}(\varphi)+\int_{0}^{t}\rho_{s}(A\varphi)ds\\
 &  & \quad+\int_{S_{0}\times[0,t]}\rho_{s}(\nabla\varphi\cdot\alpha(\cdot,u)+\varphi h(\cdot,u))Y(du\times ds),
\end{eqnarray*}
determines the unnormalized conditional distribution and the corresponding
Kushner-Stratonovich equation is 
\begin{eqnarray*}
\pi_{t}\varphi & = & \frac{\rho_{t}(\varphi)}{\rho_{t}(1)}\\
 & = & \pi_{0}\varphi+\int_{0}^{t}\pi_{s}A\varphi ds\\
 &  & \qquad+\int_{S_{0}\times[0,t]}\Big(\pi_{s}(\nabla\varphi\cdot\alpha(\cdot,u)+\varphi h(\cdot,u))-\pi_{s}\varphi\pi_{s}h(\cdot,u)\Big)Y(du\times ds)\\
 &  & \qquad+\int_{0}^{t}\int_{S_{0}}\Big(\pi_{s}\varphi\pi_{s}h(\cdot,u)-\pi_{s}(\nabla\varphi\cdot\alpha(\cdot,u)+\varphi h(\cdot,u))\Big)\pi_{s}h(\cdot,u)\mu_{0}(du)ds\\
 & = & \pi_{0}\varphi+\int_{0}^{t}\pi_{s}A\varphi ds\\
 &  & \qquad+\int_{S_{0}\times[0,t]}\Big(\pi_{s}(\nabla\varphi\cdot\alpha(\cdot,u)+\varphi h(\cdot,u))-\pi_{s}\varphi\pi_{s}h(\cdot,u)\Big)\tilde{Y}(du\times ds)
\end{eqnarray*}
where 
\[
\tilde{Y}(C,t)=Y(C,t)-\int_{0}^{t}\int_{C}\pi_{s}h(\cdot,u)\mu_{0}(du)ds.
\]
\end{theorem}

\subsection{Cluster detection in spatial point processes}

\label{mod3} 

\subsubsection{The model}

The following example is a simplified version of the models considered
in \citep{Wu09,Wu07}. Natural settings in which this problem might
arise include internet packets that form a malicious attack on a computer
system, financial transactions that form a collusive trading scheme,
and, the example consider in \citep{Wu09}, earthquakes that form
a single seismic event.

Let $E$ be a measurable space and ${\cal C}(E)$ be the collection
of counting measures on $E$ and ${\cal C}(E\times[0,\infty))$ the
collection of counting measures on $E\times[0,\infty)$. The observations
form a marked point process $O$ with marks in $E$, that is for $t\geq0$,
$O(\cdot,t)\in{\cal C}(E)$, that include the cluster $C$ (the signal)
and the noise $N$ 
\[
O(A,t)=N(A,t)+C(A,t),\quad A\in{\cal B}(E),t\geq0.
\]
For simplicity, we will assume the $O(E,t)$ is finite for all $t$.

For $\xi_{1}$ and $\xi_{2}$ independent Poisson random measures
on $E\times[0,\infty)\times[0,\infty)$ with mean measure $ $$\nu\times\ell\times\ell$,
$\ell$ denoting Lebesgue measure, $\gamma$ a nonnegative function
on $E$, and $\lambda:E\times{\cal C}(E\times[0,\infty))\rightarrow[0,\infty)$,
$N$ and $C$ can be written as solutions of 
\begin{eqnarray}
N(A,t) & = & \int_{A\times[0,\infty)\times[0,t]}{\bf 1}_{[0,\gamma(u)]}(v)\xi_{1}(du\times dv\times ds)\label{pert}\\
C(A,t) & = & \int_{A\times[0,\infty)\times[0,t]}{\bf 1}_{[0,\lambda(u,\eta_{s-})]}(v)\xi_{2}(du\times dv\times ds),\nonumber 
\end{eqnarray}
where $\eta$ is given by 
\[
\eta_{t}(A\times[0,r])=\int_{A\times[0,t]}{\bf 1}_{A}(u){\bf 1}_{[0,r]}(s)C(du\times ds),\quad A\in{\cal B}(E),r\in[0,t],
\]
that is, $\eta_{t}$ is the collection of points in the cluster up
to time $t$. The noise, $N$, is a space-time Poisson process. We
assume there exists $\lambda_{0}$ such that $\lambda(u,\eta)\leq\lambda_{0}(u)$
for all $\eta$ and that $\int_{E}\gamma(u)\nu(du)<\infty$ and $\int_{E}\lambda_{0}(u)\nu(du)<\infty$.
Assuming $N(E,0)=C(E,0)=0$, these assumptions assure that $N(E,t)$
is Poisson distributed with mean $t\int_{E}\gamma(u)\nu(du)$ and
that $C(E,\cdot)$ is dominated by a Poisson process.

Of course, if $E$ is a finite set, this model is essentially a filtering
model for counting processes as studied by Bremaud \citep{Bre81}.

\subsubsection{The reference probability space}

On $(\Omega,{\cal F},Q)$, let $N$ and $C$ be independent Poisson
random measures with mean measures $\nu_{0}(du\times ds)=\gamma(u)\nu(du)ds$
and $\nu_{1}(du\times ds)=\lambda_{0}(u)\nu(du)ds$ respectively.
At each point $(u,t)$ in $O=N+C$, let $\theta(u,t)=1$, if $(u,t)\in C$
and $\theta(u,t)=0$ otherwise. Then $Q\{\theta(u,t)=1|O\}=\frac{\lambda_{0}(u)}{\lambda_{0}(u)+\gamma(u)}$,
and hence, under $Q$, what is known, $O$, is independent of what
is not known, $\{\theta(u,t):(u,t)\in O\}$. Note also that the $\theta(u,t)$
are independent of each other, and $\eta$ and $\theta$ are related
by 
\begin{equation}
\eta_{t}(A\times[0,r])=\sum_{(u,s)\in O(A\times[0,r])}\theta(u,s),\quad A\in{\cal B}(E),r\in[0,t].\label{etatht}
\end{equation}

Under $Q$, 
\begin{equation}
\tilde{C}(A,t)=C(A,t)-\int_{0}^{t}\int_{A}\lambda_{0}(u)\nu(du)ds,\quad A\in{\cal B}(E),t\geq0,\label{ccent}
\end{equation}
is a martingale random measure. In particular, for each $A\in{\cal B}(E)$,
(\ref{ccent}) is a $\{{\cal F}_{t}\}$-martingale for ${\cal F}_{t}=\sigma(C(A,s),N(A,s):s\leq t,A\in{\cal B}(E))$.

Let $L$ satisfy 
\begin{eqnarray*}
L(t) & = & 1+\int_{E\times[0,t]}\left(\frac{\lambda(u,\eta_{s-})}{\lambda_{0}(u)}-1\right)L(s-)\left(C(du\times ds)-\lambda_{0}(u)\nu(du)ds\right)\\
 & = & 1+\int_{E\times[0,t]}\left(\frac{\lambda(u,\eta_{s-})}{\lambda_{0}(u)}-1\right)L(s-)\theta(u,t)O(du\times ds)\\
 &  & \qquad-\int_{E\times[0,t]}\left(\lambda_{0}(u,\eta_{s-})-\lambda_{0}(u)\right)L(s)\nu(du)ds.
\end{eqnarray*}
At each point $(u,s)\in C$ 
\[
L(s)=\frac{\lambda(u,\eta_{s-})}{\lambda_{0}(u)}L(s-),
\]
so $L$ is nonnegative and $L$ is an $\{{\cal F}_{t}\}$-martingale
under $Q$.

Define $dP_{|{\cal F}_{t}}=L(t)dQ_{|{\cal F}_{t}}$. Under $P$, for
all $A$, by the results in Section \ref{sectmcm} below, 
\begin{equation}
C(A,t)-\int_{A\times[0,t]}\lambda(u,\eta_{s})\nu(du)ds\label{MGPC}
\end{equation}
is a local martingale, $N$ is independent of $C$, and is a Poisson
random measure with mean measure $\nu_{0}$, that is, under $P$,
$(N,C)$ has the distribution of the solution of (\ref{pert}).

\subsubsection{Filtering equations}

Observing that 
\begin{eqnarray*}
\varphi(\eta_{t})L(t) & = & \varphi(\eta_{0})+\int_{E\times[0,t]}\left(\varphi(\eta_{s-}+\delta_{(u,s)})\frac{\lambda(u,\eta_{s-})}{\lambda(u)}-\varphi(\eta_{s-})\right)L(s-)C(du\times ds)\\
 &  & \qquad-\int_{E\times[0,t]}\varphi(\eta_{s})\left(\lambda(u,\eta_{s})-\lambda(u)\right)L(s)\nu(du)ds\\
 & = & \varphi(\eta_{0})+\int_{E\times[0,t]}\left(\varphi(\eta_{s-}+\delta_{(u,s)})\frac{\lambda(u,\eta_{s-})}{\lambda(u)}-\varphi(\eta_{s-})\right)L(s-)\theta(u,s)O(du\times ds)\\
 &  & \qquad-\int_{E\times[0,t]}\varphi(\eta_{s})\left(\lambda(u,\eta_{s})-\lambda(u)\right)L(s)\nu(du)ds,
\end{eqnarray*}
and averaging over independent choices of the $\theta(u,s)$ with
$\eta$ given by (\ref{etatht}), we have

\begin{theorem} The unnormalized conditional distributions satisfies
\begin{eqnarray*}
 &  & \rho_{t}(\varphi)=\rho_{0}(\varphi)-\int_{E\times[0,t]}\rho_{s}\left(\varphi(\cdot)\left(\lambda(u,\cdot)-\lambda(u)\right)\right)\nu(du)ds\\
 &  & \qquad+\int_{E\times[0,t]}\rho_{s-}(\varphi\left(\cdot+\delta_{(u,s)}\right)\frac{\lambda(u,\cdot)}{\lambda(u)}-\varphi(\cdot),)\frac{\lambda(u)}{\lambda(u)+\gamma(u)}O(du\times ds)
\end{eqnarray*}
and 
\begin{eqnarray*}
 &  & \pi_{t}\varphi=\pi_{0}\varphi\\
 &  & \qquad+\int_{E\times[0,t]}\frac{\pi_{s-}\left(\varphi\left(\cdot+\delta_{(u,s)}\right)\lambda(u,\cdot)\right)-\pi_{s-}\lambda(u,\cdot)\pi_{s-}\varphi}{\pi_{s-}\lambda(u,\cdot)+\gamma(u)}O(du\times ds)\\
 &  & \qquad-\int_{E\times[0,t]}\left(\pi_{s}\left(\varphi(\cdot)\lambda(u,\cdot)\right)-\pi_{s}\varphi\pi_{s}\lambda(u,\cdot)\right)\nu(du)ds
\end{eqnarray*}
\end{theorem}

\begin{remark}In most settings, the difficulty of computing the distribution
of $2^{O(E,t)}$ possible states would be prohibitive. The compromise
in \citep{Wu09} is to structure the model in such a way that it is
possible to compute $\pi_{t}\varphi=\mathbb{E}^{P}\left[\varphi(\eta_{s})|{\cal F}_{s}\right]$
for a ``small'' collection of $\varphi$.

Suppose one observes $u_{i}$ at time $\tau_{i}$ and $y_{i}=(u_{i},\tau_{i})$.
Let 
\[
\theta(y_{i})(\cdot)={\bf 1}_{\{y_{i}\mbox{{\rm \ is a point in the cluster}}\}}
\]
and 
\[
\theta_{0}(y_{i})(\cdot)={\bf 1}_{\{y_{i}\mbox{{\rm \ is the latest point in the cluster}}\}}.
\]
One needs to be able to evaluate 
\[
\pi_{t}\lambda(u,\cdot)
\]
which is accomplished under a Markov scenario.

Consider 
\[
\lambda(u,\eta_{t})=\sum_{i=1}^{O(E,t)}\lambda(u,y_{i})\theta_{0}(y_{i})+\epsilon(u),
\]
Then the goal is to obtain a closed system for 
\[
\pi_{t}\theta_{0}(y_{i}),\quad\pi_{t}\theta(y_{i}),\quad\pi_{t}\theta(y_{i})\theta_{0}(y_{j}).
\]
See Section 1.3 of \citep{Wu07}. \end{remark}

\section{Convergence of first order discretizations}

\label{subsec: diffusion in GWN} 

The solution of the stochastic filtering problem depends on both the
signal and on the observation process. However the manner in which
it depends on both ingredients is different. Let us take, as an example,
the framework described in Section 2.1. By Kallianpur-Striebel's formula,
we have that 
\[
\pi_{t}(\varphi)=\mathbb{E}^{P}\left[\varphi(X(t))|{\cal F}_{t}^{Y}\right]=\frac{\mathbb{E}^{Q}\left[\varphi(X(t))L(t)|{\cal F}_{t}^{Y}\right]}{\mathbb{E}^{Q}\left[L(t)|{\cal F}_{t}^{Y}\right]},
\]
where $Q$ is the reference measure. Under $Q$, $X$ and $Y$ are
independent and 
\begin{equation}
L(t)=\exp\left\{ \int_{0}^{t}h^{T}(X(s))dY(s)-\frac{1}{2}\int_{0}^{t}\left|h\right|^{2}(X(s))ds\right\} .\label{lt}
\end{equation}
Because of the conditioning with respect to the observation $\sigma$-algebra,
the observation path can be assumed to be fixed to a particular realisation
(the one that is actually observed). Although this fact is not immediately
clear from the expression appearing in \eqref{lt}, a simple integration
by parts of the stochastic integral in \eqref{lt} can justify the
observation path dependence of $L_{t}$, see e.g. \citep{ClaCri05,crisan2021pathwise}
for further details. The signal process enters into the solution of
the filtering problem through its law. The pathwise behavior of the
signal plays no role; it is just its law that is needed to compute
$\pi_{t}(\varphi)$. Any numerical resolution of the solution of the
filtering problem involves a discretization of the observation path
as well as the approximation of the law of the signal. The error in
the numerical approximations of $\pi_{t}(\varphi)$ will then depend
on the observation path discretization error as well as the error
due to the approximation of the signal. In the following, we will
analyze the error due to the observation path discretization by exploiting
the particle representation of the various quantities involved. In
this section we show, under very general conditions, that the discretization
error tends to $0$ as the time discretization mesh converges to $0$,
and in the next section we compute the order of convergence as well
as the leading error coefficient.

Typically, the observation data is recorded at discrete times, and
only these data are made available and used. For example, if the set
of data $\{Y_{\frac{k}{n}},k\ge0\}$ is available, we can use the
approximation 
\begin{equation}
\pi_{t}^{n}(\varphi)=\frac{\mathbb{E}^{Q}\left[\varphi(X(t))L^{n}(t)|{\cal F}_{t}^{Y}\right]}{\mathbb{E}^{Q}\left[L^{n}(t)|{\cal F}_{t}^{Y}\right]},\ \ \label{pin}
\end{equation}
where for $t=k/n$, 
\begin{equation}
L^{n}(k/n)=\exp\left(\sum_{i=0}^{k-1}\left(h^{T}\left(X(i/n)\right)\left(Y\left((i+1)/n\right)-Y\left(i/n\right)\right)-\frac{1}{2n}|h|^{2}((X(i/n))\right)\right).\label{ltn}
\end{equation}
More generally, we can embed the above approximation into the usual
continuous time version. Let $\tau_{n}\left(s\right)=j\frac{1}{n}$
for $s\in\left[j\frac{1}{n},\left(j+1\right)\frac{1}{n}\right)$,
and let $L^{n}(t)$ be given by 
\begin{equation}
L^{n}(t)=\exp\left(\int_{0}^{t}h^{T}(X(\tau_{n}\left(s\right)))dY(s)-\frac{1}{2}\int_{0}^{t}|h|^{2}(X(\tau_{n}\left(s\right)))ds\right)\label{ltn2}
\end{equation}
and then re-write \eqref{pin} as 
\begin{equation}
\pi_{t}^{n}(\varphi)=\frac{\rho_{t}^{n}(\varphi)}{\rho_{t}^{n}(1)},\ \ \label{pin2}
\end{equation}
where $\rho_{t}^{n}$ is the Picard approximation with time step $1/n$
for the unnormalized conditional distribution $\rho$ 
\[
\rho_{t}^{n}\left(\varphi\right)\triangleq E^{Q}\left[\left.\varphi\left(X\left(t\right)\right)L^{n}\left(t\right)\right|{\cal F}_{t}^{Y}\right].
\]

\begin{proposition}\label{Prop: PicardParticleRepresentation1} For
all $\varphi\in C_{P}\left(\mathbb{R}^{d_{X}}\right)$, we have that
\[
\lim_{n\rightarrow\infty}\rho_{t}^{n}\left(\varphi\right)=\rho_{t}\left(\varphi\right).
\]
Moreover $\lim_{n\rightarrow\infty}\rho^{n}=\rho$ and $\lim_{n\rightarrow\infty}\pi^{n}=\pi$
as measure valued processes. \end{proposition}

\begin{proof}\ As announced, we use the particle representation
for the respective measure valued processes $\rho^{n}$, $\rho$,
$\pi^{n}$ and $\pi$. Let $X_{1},X_{2},\ldots$ be i.i.d. copies
of $X$ that are independent of $Y$ under $Q$, and let $L_{k}^{n}$
be the exponentials corresponding to $X_{1},X_{2},\ldots$ defined
as above, i.e., 
\begin{eqnarray*}
L_{k}^{n}(t) & = & \exp\left(\int_{0}^{t}h^{T}(X_{k}(\tau_{n}\left(s\right)))dY(s)-\frac{1}{2}\int_{0}^{t}|h|^{2}(X_{k}(\tau_{n}\left(s\right)))ds\right)\\
L_{k}(t) & = & \exp\left(\int_{0}^{t}h^{T}(X_{k}(s))dY(s)-\frac{1}{2}\int_{0}^{t}|h|^{2}(X_{k}(s))ds\right).
\end{eqnarray*}
Then, the particle representations give 
\[
\rho_{t}^{n}\left(\varphi\right)=\lim_{N\rightarrow\infty}\frac{1}{N}\sum_{k=1}^{N}\varphi\left(X_{k}\left(t\right)\right)L_{k}^{n}\left(t\right)
\]
and 
\[
\rho_{t}\left(\varphi\right)=\lim_{N\rightarrow\infty}\frac{1}{N}\sum_{k=1}^{N}\varphi\left(X_{k}\left(t\right)\right)L_{k}\left(t\right),
\]
where $\varphi$ is a continuous function with at most polynomial
growth, and convergence is assured by the exchangeability of $\{(L_{k}^{n},X_{k})\}$
and $\{(L_{k}^{n},X_{k})\}$.

We apply Lemma \ref{exlimproc} with $Z^{n}=\left(\left(X_{1},L_{1}^{n}\right),\left(X_{2},L_{2}^{n}\right),...\right).$
Note that, in this example, $N_{n}=\infty$. Thanks to Remark \ref{cEinf}
it suffices to check the convergence of $Z^{n}$ to $Z$ in $C_{\left(\mathbb{R}^{d_{X}}\mathbf{\textnormal{\ensuremath{\times\textnormal{ [0}},\ensuremath{\infty})}}\right)^{\infty}}\textnormal{[0},\infty)$
in probability. For fixed $k$, one has that 
\[
\mathbb{E}^{Q}\left[\sup_{0\leq s\leq t}\left|L_{k}^{n}-L_{k}\right|^{2}\right]\rightarrow_{n\rightarrow\infty}0.
\]
 for all $t\geq0$. This follows from the inequality $\left|e^{y}-e^{x}\right|\leq\frac{\left(e^{x}+e^{y}\right)}{2}\left|x-y\right|$
combined with similar estimates as those in Lemmas 3.6 and 3.9 in
\citep{CriO-L20}. Therefore, $\left(X_{k},L_{k}^{n}\right)$ converges
in probability to $\left(X_{k},L_{k}\right)$ when $n$ tends to infinity.
Lemma \ref{exlimproc} b) yields that 
\[
\lim_{n\rightarrow\infty}V^{n}=\lim_{n\rightarrow\infty}\lim_{N\rightarrow\infty}\frac{1}{N}\sum_{k=1}^{N}\delta_{\left(X_{k},L_{k}^{n}\right)}=\lim_{N\rightarrow\infty}\lim_{n\rightarrow\infty}\frac{1}{N}\sum_{k=1}^{N}\delta_{\left(X_{k},L_{k}\right)}=V,
\]
in $C_{\mathcal{P}\left(\mathbb{R}^{d_{X}}\mathbf{\textnormal{\ensuremath{\times\textnormal{[0}},\ensuremath{\infty})}}\right)}\textnormal{ [0},\infty)$.

The moment estimates on $X_{k}$, $L_{k}^{n}$, and $L_{k}$ ensure
that for each $\psi\in C_{P}\left(\mathbb{R}^{d_{X}}\right)$ and
$m>0$ and each $T>0$, 
\[
\sup_{n}\mathbb{E}[\sup_{t\leq T}\psi(X_{k}(t))(L_{k}^{n}(t)^{m}+1)]=\sup_{n}\mathbb{E}[\sup_{t\leq T}\int_{\mathbb{R}^{d_{X}}\times[0,\infty)}\psi(x)(a^{m}+1)V_{n}(dx\times da,t)<\infty.
\]
Then keeping in mind that 
\[
\rho_{t}^{n}(\varphi)=\int_{\mathbb{R}^{d_{X}}\times[0,\infty)}a\varphi(x)V^{n}(dx\times da,t)
\]
and 
\[
\rho_{t}(\varphi)=\int_{\mathbb{R}^{d_{X}}\times[0,\infty)}a\varphi(x)V(dx\times da,t),
\]
the result follows by applying Lemma \ref{imgcnv}. In particular,
for all $\varphi\in C_{P}\left(\mathbb{R}^{d_{X}}\right)$, we have
that 
\[
\lim_{n\rightarrow\infty}\rho_{t}^{n}\left(\varphi\right)=\rho_{t}\left(\varphi\right).
\]
with the convergence of the measure valued processes $\lim_{n\rightarrow\infty}\rho^{n}=\rho$
and $\lim_{n\rightarrow\infty}\pi^{n}=\pi$ being an immediate consequence
of the above and of the Kallianpur-Striebel formula. \end{proof}

A similar result can be obtained for the second framework (Spatial
observations with additive white noise). The above convergence result
does not give an estimate of the order of convergence. This is not
possible under the general assumptions of stated in section 2.1.2
on the functions $h,b$ and $\sigma$. However, we can do this under
more restrictive assumptions. This is the goal of the next section.

\section{Leading error coefficient for the Picard discretization}

\label{secterr}

In this section, we are using the same framework as in sections \ref{mod1}
and \ref{subsec:  diffusion in  GWN}, as well as the same notation
introduced therein. In addition, we will require that $\sigma,b,h\in C_{b}^{6}\left(\mathbb{R}^{d_{X}}\right)$.
Heuristically, the main goal of this section is to show that 
\[
\rho_{t}^{n}=\rho_{t}-\frac{1}{n}U_{t}+o\left(\frac{1}{n}\right),\qquad t\in[0,\infty),
\]
where $U$ is a process characterized as the solution of a certain stochastic evolution equation.\footnote{A similar expansion holds for $\pi_{t}$ using a straightforward application
of the Kallianpur-Striebel's formula.} The exact statement of the result is contained in Theorem \ref{th:lr}
below. The main technical tool to do this is, again, the particle
representations of the various processes involved. To be more precise,
let $C_{P}^{k}$ be the space of $k$-differentiable functions
with at most polynomial growth. Then, for $\varphi\in C_{P}^{k}$
consider the quantities

\begin{eqnarray*}
E_{t}^{k,n}\left(\varphi\right) &  & \triangleq n\varphi\left(X_{k}\left(t\right)\right)\left(L_{k}\left(t\right)-L_{k}^{n}\left(t\right)\right),\\
U_{t}^{N,n}\left(\varphi\right) &  & \triangleq\frac{1}{N}\sum_{k=1}^{N}E_{t}^{k,n}\left(\varphi\right)
\end{eqnarray*}
and note that, combining the results in the previous sections, we
get 
\begin{eqnarray*}
\lim_{N\rightarrow\infty}U_{t}^{N,n}\left(\varphi\right) & = & \mathbb{E}^{Q}\left[\left.\varphi\left(X\left(t\right)\right)n\left(L\left(t\right)-L^{n}\left(t\right)\right)\right|\mathcal{F}_{t}^{Y}\right]\\
 & = & n\left(\rho_{t}\left(\varphi\right)-\rho_{t}^{n}\left(\varphi\right)\right)\triangleq U_{t}^{n}\left(\varphi\right).
\end{eqnarray*}
The goal is to find an evolution equation for the limit of $U^{n}$
when $n$ tends to infinity. This is attained in Theorem \ref{th:lr}.
Let us introduce first two preliminary results. The first is related
to a martingale process that will converge to a Brownian motion as
$n$ tends to infinity. The second is the evolution equation for the
process $U^{n}$.

Consider the sequence of processes $R^{n}=\left\{ R_{t}^{n}\right\} _{t\geq0},n\in\mathbb{N},$
defined by 
\begin{equation}
R^{n}\left(t\right)\triangleq2\sqrt{3}\int_{0}^{t}\left(n\left(s-\tau_{n}\left(s\right)\right)-\frac{1}{2}\right)dY\left(s\right).\label{eq:defRn}
\end{equation}

\begin{lemma} \label{independentBM}$\left(Y,R^{n}\right)$ converges
in distribution to $\left(Y,\mathcal{R}\right)$, where $\mathcal{R}$
is a $d_{Y}$-dimensional standard Brownian motion independent of
$Y$ and the $B_{k}$. \end{lemma}

\begin{proof} The process $\left(Y,R^{n}\right)$ is a $2d_{Y}$-dimensional
martingale as $R^{n}$ is a stochastic integral with respect to the
Brownian motion $Y$. Moreover, observe that 
\begin{eqnarray}
\left[Y^{i},Y^{j}\right]_{t} & = & [R^{n,i},R^{n,j}]_{t}=[Y^{i},R^{n,j}]_{t}=0,\qquad i\neq j,\nonumber \\
\left[R^{n,i}\right]_{t} & = & 12\int_{0}^{t}\left(n\left(s-\tau_{n}\left(s\right)\right)-\frac{1}{2}\right)^{2}ds=t+\mathcal{O}\left(n^{-1}\right),\label{eq:QuadVarRn}\\
\left[Y^{i},(R^{n})^{i}\right]_{t} & = & 2\sqrt{3}\int_{0}^{t}\left(n\left(s-\tau_{n}\left(s\right)\right)-\frac{1}{2}\right)ds=\mathcal{O}\left(n^{-1}\right).\nonumber 
\end{eqnarray}
The result follows by the martingale central limit theorem, for example,
Theorem 1.4, Chapter 7 in \citep{EK86}. \end{proof}

Let $\left\{ S_{n}\right\} _{n\geq1}$ be a sequence of real-valued
random processes. In what follows we will use the notation $S_{n}=\mathcal{O}\left(n^{-p}\right)$
for some $p\geq0$ to indicate that 
\[
\mathbb{E}^{Q}\left[\sup_{0\leq s\leq t}\left|S_{n}\left(s\right)\right|^{2}\right]^{1/2}\leq\frac{C\left(t\right)}{n^{p}},
\ \ t\ge 0,
\]
for some positive constant $C\left(t\right).$

\begin{proposition}\label{Pr:Un} For each $\varphi\in C_{P}^{4}$
and $n\in\mathbb{N}$, the process $U^{n}$ satisfies, the following
approximate evolution equation: 
\begin{eqnarray}
U_{t}^{n}\left(\varphi\right) & = & \int_{0}^{t}U_{s}^{n}\left(A\varphi\right)ds+\int_{0}^{t}U_{s}^{n}\left(\varphi h^{T}\right)dY\left(s\right)+\frac{1}{2}\int_{0}^{t}\rho_{s}^{n}\left(\varphi Ah^{T}\right)dY\left(s\right)\label{eqUn}\\
 &  & +\frac{1}{2\sqrt{3}}\int_{0}^{t}\rho_{s}^{n}\left(\varphi Ah^{T}\right)dR^{n}\left(s\right)+\frac{1}{2\sqrt{3}}\sum_{i=1}^{d_{Y}}\int_{0}^{t}\rho_{\tau_{n}\left(s\right)}^{n}\left(\mathrm{tr}\left(\widetilde{O}_{\sigma,\varphi,h_{i}}\right)\right)dR^{n,i}\left(s\right)\nonumber \\
 &  & +\frac{1}{2}\sum_{i=1}^{d_{Y}}\int_{0}^{t}\rho_{\tau_{n}\left(s\right)}^{n}\left(\mathrm{tr}\left(\widetilde{O}_{\sigma,\varphi,h_{i}}\right)\right)dY^{i}\left(s\right)+\Gamma_{t}^{n}(\varphi),\nonumber 
\end{eqnarray}
where $\Gamma^{n}(\varphi)$ is a process satisfying
$\Gamma^{n}(\varphi)=\mathcal{O}\left(n^{-1/2}\right)$ for all $\varphi\in C_{P}^{3}$,
and 
\[
\widetilde{O}_{\sigma,\varphi,h_{i}}\left(x\right)=\frac{1}{2}\left(O_{\sigma,\varphi,h_{i}}\left(x\right)+O_{\sigma,\varphi,h_{i}}^{T}\left(x\right)\right),
\]
with 
\[
O_{\sigma,\varphi,h_{i}}\left(x\right)=\sigma^{T}\left(x\right)\nabla\varphi\left(x\right)\nabla h_{i}^{T}\left(x\right)\sigma\left(x\right).
\]

\end{proposition}
We give the proof of Proposition \ref{Pr:Un} in Section \ref{section:Pr:Un}
below. 

\begin{lemma}\label{Uest} For each $\varphi\in C_{P}^{4}$ and $t>0$,
we have that 
\[
\sup_{n\in\mathbb{N}}\mathbb{E}^{Q}\left[\sup_{0\leq s\leq t}\left|U_{s}^{n}\left(\varphi\right)\right|^{2}\right]\leq C\left(t,\varphi\right),
\]
for some positive constant $C\left(t,\varphi\right)$.

\end{lemma}

\begin{proof} Theorem 2.3 in \citep{CriO-L20}, with $m=1$, states
that, for any $\varphi\in C_{P}^{2}$, there exists a constant $C$ (not depending on $n$ but possibly
on $t,\varphi,b$, $\sigma$, and $h$) such that 
\[
\sup_{n}\sup_{0\leq s\leq t}\mathbb{E}^{Q}\left[\left|U_{s}^{n}\left(\varphi\right)\right|^{2}\right] <\infty,
\]

This estimate clearly yields that, for any $t\ge 0$ and $\varphi\in C_{P}^{4}$ (we need the additional smoothness to ensure that $A\varphi\in C_{P}^{4}$) 
\[
\bar{C}_t:=\sup_{n}\left(
\sup_{0\leq s\leq t}\mathbb{E}^{Q}\left[\left|U_{s}^{n}\left(A\varphi\right)\right|^{2}\right] + 
\sup_{0\leq s\leq t}\mathbb{E}^{Q}\left[\left|U_{s}^{n}\left(\varphi h\right)\right|^{2}\right]\right) <\infty,
\]
By
Proposition \ref{Pr:Un}, we can write
\[
U_{t}^{n}\left(\varphi\right)=\int_{0}^{t}U_{s}^{n}\left(A\varphi\right)ds+\int_{0}^{t}U_{s}^{n}\left(\varphi h^{T}\right)dY\left(s\right)+S_{t}\left(\varphi\right),
\]
where in $S_{t}\left(\varphi\right)$ we put all the terms in equation
$\left(\ref{eqUn}\right)$ not containing $U^{n}.$ We deduce that  
\begin{align*}
\mathbb{E}^{Q}\left[\sup_{0\leq s\leq t}\left|U_{s}^{n}\left(\varphi\right)\right|^{2}\right] & \leq3\mathbb{E}^{Q}\left[\sup_{0\leq s\leq t}\left|\int_{0}^{s}U_{r}^{n}\left(A\varphi\right)dr\right|^{2}\right]\\
 & \quad+3\mathbb{E}^{Q}\left[\sup_{0\leq s\leq t}\left|\int_{0}^{s}U_{r}^{n}\left(\varphi h^{T}\right)dY\left(r\right)\right|^{2}\right]\\
 & \quad+3\mathbb{E}^{Q}\left[\sup_{0\leq s\leq t}\left|S_{s}\left(\varphi\right)\right|^{2}\right]\\
 & \triangleq3(I_{1}^{n}+I_{2}^{n}+I_{3}^{n}).
\end{align*}
and, therefore, it suffices to bound $I_{1}^{n},I_{2}^{n}$ and $I_{3}^{n}$
to justify the claim.
Using Cauchy-Schwarz' inequality and Fubini's theorem we obtain
\[
I_{1}^{n}\leq\mathbb{E}^{Q}\left[\sup_{0\leq s\leq t}s\int_{0}^{s}\left|U_{r}^{n}\left(A\varphi\right)\right|^{2}dr\right]\leq t\int_{0}^{t}\mathbb{E}^{Q}\left[\left|U_{r}^{n}\left(A\varphi\right)\right|^{2}\right]dr\leq t^{2}\bar{C}_t,
\]
All the remaining terms are stochastic integrals with respect to continuous martingales and can be controlled by means of Doob's maximal inequality. \end{proof}

We are now in a position to state and prove the main result of this
section.

\begin{theorem}\label{th:lr} Let $\{U^{n}\}_{n\geq1}$ be the sequence
of processes given by 
\[
U_{t}^{n}=n\left(\rho_{t}-\rho_{t}^{n}\right),\qquad t\in[0,\infty).
\]
Then for each $\varphi\in C_{P}^{8}$, the sequence $\left\{ U^{n}\left(\varphi\right),U^{n}\left(A\varphi\right),U^{n}\left(\varphi h^{i}\right),i=1,\ldots, d_Y\right\} $
is relatively compact in $C_{\mathbb{R}^{2+d_{Y}}}[0,\infty)$, and every limit point
satisfies 
\begin{eqnarray}
U_{t}\left(\varphi\right) & = & \int_{0}^{t}U_{s}\left(A\varphi\right)ds+\int_{0}^{t}U_{s}\left(\varphi h^{T}\right)dY\left(s\right)+\frac{1}{2}\int_{0}^{t}\rho_{s}\left(\varphi Ah^{T}\right)dY\left(s\right)\label{eqU}\\
 &  & \quad+\frac{1}{2\sqrt{3}}\int_{0}^{t}\rho_{s}\left(\varphi Ah^{T}\right)d\mathcal{R}\left(s\right)+\frac{1}{2\sqrt{3}}\sum_{i=1}^{d_{Y}}\int_{0}^{t}\rho_{s}\left(\mathrm{tr}\left(\widetilde{O}_{\sigma,\varphi,h_{i}}\right)\right)d\mathcal{R}^{i}\left(s\right)\nonumber \\
 &  & \quad+\frac{1}{2}\sum_{i=1}^{d_{Y}}\int_{0}^{t}\rho_{s}\left(\mathrm{tr}\left(\widetilde{O}_{\sigma,\varphi,h_{i}}\right)\right)dY^{i}\left(s\right),\nonumber 
\end{eqnarray}
where $\mathcal{R}$ is a Brownian motion independent of $Y$ and
all $B_{k}$. \end{theorem}


\begin{proof} 
Fix  $\varphi\in C_{P}^{6}$.
This assumption along with the assumptions on $\sigma$, $b$, and
$h$, assure that $A\varphi\in C_{P}^{4}$ and hence that the estimate
in Lemma \ref{Uest} applies. Using this estimate on the integrands
in the first two integrals on the right of (\ref{eqUn}), we see that
these integrals are relatively compact in $C_{\mathbb{R}}[0,\infty)$.
The remaining terms on the right converge by the convergence of $\rho^{n}$.
Consequently, $\{U^{n}(\varphi)\}$ is relatively compact in $C_{\mathbb{R}}[0,\infty). $  Moreover, if we take $\varphi\in C_{P}^{8}$, $A\varphi\in C_{P}^{6}$, so the
integrands in the first two terms are relatively compact, and (\ref{eqUn})
is satisfied for any limit point. \end{proof} 

\subsection{Proof of Proposition \ref{Pr:Un}}

\label{section:Pr:Un}

We can combine $\left(\ref{geq1}\right)$ and 
\begin{align}
\varphi\left(X_{k}\left(t\right)\right)L_{k}^{n}\left(t\right) & =\varphi\left(X_{k}\left(0\right)\right)+\int_{0}^{t}L_{k}^{n}\left(s\right)\nabla\varphi^{T}\sigma\left(X_{k}\left(s\right)\right)dB_{k}\left(s\right)\nonumber \\
 & +\int_{0}^{t}L_{k}^{n}\left(s\right)A\varphi\left(X_{k}\left(s\right)\right)ds+\int_{0}^{t}L_{k}^{n}\left(s\right)\varphi h^{T}\left(X_{k}\left(s\right)\right)dY\left(s\right)\nonumber \\
 & -\int_{0}^{t}L_{k}^{n}\left(s\right)\varphi\left(X_{k}\left(s\right)\right)\left\{ h\left(X_{k}\left(s\right)\right)-h\left(X_{k}\left(\tau_{n}\left(s\right)\right)\right)\right\} ^{T}dY\left(s\right).\label{eq:Particle-Fi*Zn}
\end{align}
to write a more convenient expression for $E_{t}^{k,n}\left(\varphi\right)$,
that is, 
\begin{align*}
E_{t}^{k,n}\left(\varphi\right) & =n\varphi\left(X_{k}\left(t\right)\right)\left(L_{k}\left(t\right)-L_{k}^{n}\left(t\right)\right)\\
 & =\int_{0}^{t}n\left\{ L_{k}\left(s\right)-L_{k}^{n}\left(s\right)\right\} \nabla\varphi^{T}\sigma\left(X_{k}\left(s\right)\right)dB_{k}\left(s\right)\\
 & \quad+\int_{0}^{t}n\left\{ L_{k}\left(s\right)-L_{k}^{n}\left(s\right)\right\} A\varphi\left(X_{k}\left(s\right)\right)ds\\
 & \quad+\int_{0}^{t}n\left\{ L_{k}\left(s\right)-L_{k}^{n}\left(s\right)\right\} \varphi h^{T}\left(X_{k}\left(s\right)\right)dY\left(s\right)\\
 & \quad+\sum_{i=1}^{d_{Y}}\int_{0}^{t}nL_{k}^{n}\left(s\right)\varphi\left(X_{k}\left(s\right)\right)dM_{k}^{n,i}\left(s\right),
\end{align*}
where the processes $M_{k}^{n,i}$ are defined by 
\[
M_{k}^{n,i}\left(t\right)\triangleq\int_{0}^{t}\left(h_{i}\left(X_{k}\left(s\right)\right)-h_{i}\left(X_{k}\left(\tau_{n}\left(s\right)\right)\right)\right)dY^{i}\left(s\right).
\]
Moreover, using Itô's formula , we can write 
\begin{align*}
nM_{k}^{n,i}\left(t\right) & =\int_{0}^{t}n\left[\int_{\tau_{n}(s)}^{s}\nabla h_{i}^{T}\sigma\left(X_{k}\left(u\right)\right)dB_{k}\left(u\right)+\int_{\tau_{n}(s)}^{s}Ah_{i}\left(X_{k}\left(u\right)\right)du\right]dY^{i}\left(s\right)\\
 & =\int_{0}^{t}\int_{\tau_{n}(s)}^{s}n\nabla h_{i}^{T}\sigma\left(X_{k}\left(u\right)\right)dB_{k}\left(u\right)dY^{i}\left(s\right)\\
 & \quad+\int_{0}^{t}\int_{\tau_{n}(s)}^{s}nAh_{i}\left(X_{k}\left(u\right)\right)dudY^{i}\left(s\right)\\
 & =\int_{0}^{t}\int_{\tau_{n}(s)}^{s}n\nabla h_{i}^{T}\sigma\left(X_{k}\left(u\right)\right)dB_{k}\left(u\right)dY^{i}\left(s\right)\\
 & \quad+\int_{0}^{t}Ah_{i}\left(X_{k}\left(\tau_{n}\left(s\right)\right)\right)n\left(s-\tau_{n}\left(s\right)\right)dY^{i}\left(s\right)\\
 & \quad+\int_{0}^{t}\int_{\tau_{n}(s)}^{s}n\left\{ Ah_{i}\left(X_{k}\left(u\right)\right)-Ah_{i}\left(X_{k}\left(\tau_{n}\left(s\right)\right)\right)\right\} dudY^{i}\left(s\right)\\
 & =\int_{0}^{t}\int_{\tau_{n}(s)}^{s}n\nabla h_{i}^{T}\sigma\left(X_{k}\left(u\right)\right)dB_{k}\left(u\right)dY^{i}\left(s\right)+K_{k}^{n,i}\left(t\right)+I_{k}^{n,i}\left(t\right),
\end{align*}
where 
\begin{align*}
K_{k}^{n,i}\left(t\right) & \triangleq\int_{0}^{t}Ah_{i}\left(X_{k}\left(\tau_{n}\left(s\right)\right)\right)n\left(s-\tau_{n}\left(s\right)\right)dY^{i}\left(s\right),\\
I_{k}^{n,i}\left(t\right) & \triangleq\int_{0}^{t}\int_{\tau_{n}(s)}^{s}n\left\{ Ah_{i}\left(X_{k}\left(u\right)\right)-Ah_{i}\left(X_{k}\left(\tau_{n}\left(s\right)\right)\right)\right\} dudY^{i}\left(s\right).
\end{align*}
Since, 
\[
\left(n\left(s-\tau_{n}\left(s\right)\right)\right)dY^{i}\left(s\right)=\frac{dR^{n,i}\left(s\right)}{2\sqrt{3}}+\frac{1}{2}dY^{i}\left(s\right),
\]
we can write 
\[
K_{k}^{n,i}\left(t\right)=\frac{1}{2\sqrt{3}}\int_{0}^{t}Ah_{i}\left(X_{k}\left(s\right)\right)dR^{n,i}\left(s\right)+\frac{1}{2}\int_{0}^{t}Ah_{i}\left(X_{k}\left(s\right)\right)dY^{i}\left(s\right),
\]
and 
\begin{align*}
nM_{k}^{n,i}\left(t\right) & =\int_{0}^{t}\int_{\tau_{n}(s)}^{s}n\nabla h_{i}^{T}\sigma\left(X_{k}\left(u\right)\right)dB_{k}\left(u\right)dY^{i}\left(s\right)\\
 & \quad+\frac{1}{2\sqrt{3}}\int_{0}^{t}Ah_{i}\left(X_{k}\left(s\right)\right)dR^{n,i}\left(s\right)\\
 & \quad+\frac{1}{2}\int_{0}^{t}Ah_{i}\left(X_{k}\left(s\right)\right)dY^{i}\left(s\right)+I_{k}^{n,i}\left(s\right).
\end{align*}
Finally, we can also write 
\begin{align*}
E_{t}^{k,n}\left(\varphi\right) & =\int_{0}^{t}n\left\{ L_{k}\left(s\right)-L_{k}^{n}\left(s\right)\right\} \nabla\varphi^{T}\sigma\left(X_{k}\left(s\right)\right)dB_{k}\left(s\right)\\
 & \quad+\int_{0}^{t}n\left\{ L_{k}\left(s\right)-L_{k}^{n}\left(s\right)\right\} A\varphi\left(X_{k}\left(s\right)\right)ds\\
 & \quad+\int_{0}^{t}n\left\{ L_{k}\left(s\right)-L_{k}^{n}\left(s\right)\right\} \varphi h^{T}\left(X_{k}\left(s\right)\right)dY\left(s\right)\\
 & \quad+\sum_{i=1}^{d_{Y}}\int_{0}^{t}L_{k}^{n}\left(s\right)\varphi\left(X_{k}\left(s\right)\right)\int_{\tau_{n}(s)}^{s}n\nabla h_{i}^{T}\sigma\left(X_{k}\left(u\right)\right)dB_{k}\left(u\right)dY^{i}\left(s\right)\\
 & \quad+\sum_{i=1}^{d_{Y}}\frac{1}{2}\int_{0}^{t}L_{k}^{n}\left(s\right)\varphi Ah_{i}\left(X_{k}\left(s\right)\right)dY^{i}\left(s\right)\\
 & \quad+\sum_{i=1}^{d_{Y}}\frac{1}{2\sqrt{3}}\int_{0}^{t}L_{k}^{n}\left(s\right)\varphi Ah_{i}\left(X_{k}\left(s\right)\right)dR^{n,i}\left(s\right)\\
 & \quad+\sum_{i=1}^{d_{Y}}\int_{0}^{t}L_{k}^{n}\left(s\right)\varphi\left(X_{k}\left(s\right)\right)dI_{k}^{n,i}\left(s\right)\triangleq\sum_{i=1}^{7}A_{i}.
\end{align*}
The result follows by averaging over $1\leq k\leq N$ in the previous
equation and taking limits when $N$ tends to infinity, combined with
Lemmas \ref{exchconv}, \ref{lem: aux1} and \ref{lem: aux2}:
\begin{itemize}
\item For the term on the left hand side of the previous equation we have
\[
\lim_{N\rightarrow\infty}\frac{1}{N}\sum_{k=1}^{N}E_{t}^{k,n}\left(\varphi\right)=\lim_{N\rightarrow\infty}U_{t}^{N,n}\left(\varphi\right)=U_{t}^{n}\left(\varphi\right).
\]
\item For the term $A_{1}$ we can write 
\begin{align*}
 & \lim_{N\rightarrow\infty}\frac{1}{N}\sum_{k=1}^{N}\int_{0}^{t}n\left\{ L_{k}\left(s\right)-L_{k}^{n}\left(s\right)\right\} \nabla\varphi^{T}\sigma\left(X_{k}\left(s\right)\right)dB_{k}\left(s\right)\\
 & =\mathbb{E}^{Q}\left[\left.\int_{0}^{t}n\left\{ L\left(s\right)-L^{n}\left(s\right)\right\} \nabla\varphi^{T}\sigma\left(X\left(s\right)\right)dB\left(s\right)\right|\mathcal{F}_{t}^{Y}\right]\\
 & =\mathbb{E}^{Q}\left[\left.\mathbb{E}^{Q}\left[\left.\int_{0}^{t}n\left\{ L\left(s\right)-L^{n}\left(s\right)\right\} \nabla\varphi^{T}\sigma\left(X\left(s\right)\right)dB\left(s\right)\right|\mathcal{F}_{t}^{Y}\vee\mathcal{F}_{0}^{V}\right]\right|\mathcal{F}_{t}^{Y}\right]=0.
\end{align*}
\item For the term $A_{2}$, using Proposition 3.15 in \citep{bc2009},
we can write
\begin{align*}
 & \lim_{N\rightarrow\infty}\frac{1}{N}\sum_{k=1}^{N}\int_{0}^{t}n\left\{ L_{k}\left(s\right)-L_{k}^{n}\left(s\right)\right\} A\varphi\left(X_{k}\left(s\right)\right)ds\\
 & =\mathbb{E}^{Q}\left[\left.\int_{0}^{t}n\left\{ L\left(s\right)-L^{n}\left(s\right)\right\} A\varphi\left(X\left(s\right)\right)ds\right|\mathcal{F}_{t}^{Y}\right]\\
 & =\int_{0}^{t}n\mathbb{E}^{Q}\left[\left.\left\{ L\left(s\right)-L^{n}\left(s\right)\right\} A\varphi\left(X\left(s\right)\right)\right|\mathcal{F}_{s}^{Y}\right]ds\\
 & =\int_{0}^{t}U_{s}^{n}\left(A\varphi\right)ds.
\end{align*}
\item The terms $A_{3},A_{5}$ and $A_{6}$ are treated similarly as the
term $A_{2}$. Note that the processes $R^{n,i}$ are $\mathcal{F}_{t}^{Y}$-adapted.
\item For the term $A_{4}$ we apply Lemma \ref{lem: aux2}.
\item For the term $A_{7}$ we apply Lemma \ref{lem: aux1}.
\end{itemize}
The process $\Gamma_{t}^{n}\left(\varphi\right)=\sum_{i=1}^{d_{Y}}\left(\Phi_{t}^{n,i}\left(\varphi\right)+\Psi{}_{t}^{n,i}\left(\varphi\right)\right)$,
where $\Phi_{t}^{n,1}\left(\varphi\right)$ and $\Psi{}_{t}^{n,i}\left(\varphi\right)$
are the processes in the statement of Lemmas \ref{lem: aux1} and
\ref{lem: aux2}.

\subsection{Auxiliary lemmas}

\begin{lemma}\label{lem: aux1}

For all $i=1,...,d_{Y}$ and $\varphi$ Borel measurable with at most
polynomial growth, let 
\[
I_{k}^{n,i}\left(t\right)=\int_{0}^{t}\int_{\tau_{n}(s)}^{s}n\left\{ Ah_{i}\left(X_{k}\left(u\right)\right)-Ah_{i}\left(X_{k}\left(\tau_{n}\left(s\right)\right)\right)\right\} dudY^{i}\left(s\right).
\]
Then,
\[
\lim_{N\rightarrow\infty}\frac{1}{N}\sum_{k=1}^{N}\int_{0}^{t}L_{k}^{n}\left(s\right)\varphi\left(X_{k}\left(s\right)\right)dI_{k}^{n,i}\left(s\right)=\Phi_{t}^{n,i}\left(\varphi\right),
\]
where $\Phi_{t}^{n,i}\left(\varphi\right)=\mathcal{O}\left(n^{-1/2}\right).$

\end{lemma}

\begin{proof}

Applying Lemma \ref{exchconv} we have that 
\begin{align*}
 & \lim_{N\rightarrow\infty}\frac{1}{N}\sum_{k=1}^{N}\int_{0}^{t}L_{k}^{n}\left(s\right)\varphi\left(X_{k}\left(s\right)\right)dI_{k}^{n,i}\left(s\right)\\
 & =\mathbb{E}^{Q}\left[\left.\int_{0}^{t}L^{n}\left(s\right)\varphi\left(X\left(s\right)\right)dI^{n,i}\left(s\right)\right|\mathcal{F}_{t}^{Y}\right],
\end{align*}
where 
\[
dI^{n,i}\left(s\right)=\left(\int_{\tau_{n}(s)}^{s}n\left\{ Ah_{i}\left(X\left(u\right)\right)-Ah_{i}\left(X\left(\tau_{n}\left(s\right)\right)\right)\right\} du\right)dY^{i}\left(s\right).
\]
In what follows we will use the more compact notation 
\[
\Delta_{n}^{i}\left(u\right)\triangleq Ah_{i}\left(X\left(u\right)\right)-Ah_{i}\left(X\left(\tau_{n}\left(s\right)\right)\right),\quad\tau_{n}\left(s\right)\leq u\leq s.
\]
Note that, using Proposition 3.15 in \citep{bc2009}, we can write
\begin{align*}
 & \mathbb{E}^{Q}\left[\left.\int_{0}^{t}L^{n}\left(s\right)\varphi\left(X\left(s\right)\right)dI^{n,i}\left(s\right)\right|\mathcal{F}_{t}^{Y}\right]\\
 & =\mathbb{E}^{Q}\left[\left.\int_{0}^{t}L^{n}\left(s\right)\varphi\left(X\left(s\right)\right)\left(\int_{\tau_{n}(s)}^{s}n\Delta_{n}^{i}\left(u\right)du\right)dY^{i}\left(s\right)\right|\mathcal{F}_{t}^{Y}\right]\\
 & =\int_{0}^{t}\mathbb{E}^{Q}\left[\left.L^{n}\left(s\right)\varphi\left(X\left(s\right)\right)\left(\int_{\tau_{n}(s)}^{s}n\Delta_{n}^{i}\left(u\right)du\right)\right|\mathcal{F}_{s}^{Y}\right]dY^{i}\left(s\right)\\
 & \triangleq\Phi_{t}^{n,i}\left(\varphi\right)
\end{align*}
Moreover, using Burkholder-Davis-Gundy inequality, Jensen's inequality
for conditional expectation, the law of total expectation, Fubini's
theorem and Cauchy-Schwarz inequality we obtain
\begin{align*}
 & \mathbb{E}^{Q}\left[\sup_{0\leq t\leq T}\left|\mathbb{E}^{Q}\left[\left.\int_{0}^{t}L_{1}^{n}\left(s\right)\varphi\left(X_{1}\left(s\right)\right)dI_{1}^{n,i}\left(s\right)\right|\mathcal{F}_{t}^{Y}\right]\right|^{2}\right]\\
 & \leq\mathbb{E}^{Q}\left[\int_{0}^{T}\left|\mathbb{E}^{Q}\left[\left.L^{n}\left(s\right)\varphi\left(X\left(s\right)\right)\left(\int_{\tau_{n}(s)}^{s}n\Delta_{n}^{i}\left(u\right)du\right)\right|\mathcal{F}_{s}^{Y}\right]\right|^{2}ds\right]\\
 & \leq\mathbb{E}^{Q}\left[\int_{0}^{T}\left|L^{n}\left(s\right)\varphi\left(X\left(s\right)\right)\left(\int_{\tau_{n}(s)}^{s}n\Delta_{n}^{i}\left(u\right)du\right)\right|^{2}ds\right]\\
 & =\int_{0}^{T}\mathbb{E}^{Q}\left[\left|L^{n}\left(s\right)\varphi\left(X\left(s\right)\right)\left(\int_{\tau_{n}(s)}^{s}n\Delta_{n}^{i}\left(u\right)du\right)\right|^{2}\right]ds\\
 & \leq\int_{0}^{T}\mathbb{E}^{Q}\left[\left|L^{n}\left(s\right)\varphi\left(X\left(s\right)\right)\right|^{4}\right]^{1/2}\mathbb{E}^{Q}\left[\left|\left(\int_{\tau_{n}(s)}^{s}n\Delta_{n}^{i}\left(u\right)du\right)\right|^{4}\right]^{1/2}ds.
\end{align*}
Using Jensen's innequality and Itô's formula we get that 
\begin{align*}
\mathbb{E}^{Q}\left[\left|\left(\int_{\tau_{n}(s)}^{s}n\Delta_{n}^{i}\left(u\right)du\right)\right|^{4}\right] & \leq\left(s-\tau_{n}\left(s\right)\right)^{3}n^{4}\int_{\tau_{n}(s)}^{s}\mathbb{E}^{Q}\left[\left|\Delta_{n}^{i}\left(u\right)\right|^{4}\right]du\\
 & \leq Cn\left\{ \int_{\tau_{n}(s)}^{s}\mathbb{E}^{Q}\left[\left|\int_{\tau_{n}\left(s\right)}^{u}\nabla\left(Ah_{i}\right)^{T}\sigma\left(X\left(v\right)\right)dB\left(v\right)\right|^{4}\right]du\right.\\
 & \quad+\left.\int_{\tau_{n}(s)}^{s}\mathbb{E}^{Q}\left[\left|\int_{\tau_{n}\left(s\right)}^{u}A^{2}h_{i}\left(X\left(v\right)\right)dv\right|^{4}\right]du\right\} \\
 & =B_{1}+B_{2}
\end{align*}
Due to the hypothesis on $\sigma,b$ and $h$ we have that $\nabla\left(Ah_{i}\right)^{T}\sigma$
has at most polynomial growth, which combined with the bound $\left(\ref{momest1}\right)$
yields that 
\[
\mathbb{E}^{Q}\left[\sup_{0\leq s\leq T}\left|\nabla\left(Ah_{i}\right)^{T}\sigma\left(X\left(s\right)\right)\right|^{4}\right]=C\left(T,h,b,\sigma\right)<\infty.
\]
Therefore, using Burkholder-Davis-Gundy innequality we obtain 
\begin{align*}
B_{1} & \leq Cn\int_{\tau_{n}(s)}^{s}\mathbb{E}^{Q}\left[\left|\int_{\tau_{n}\left(s\right)}^{u}\left|\nabla\left(Ah_{i}\right)^{T}\sigma\left(X\left(v\right)\right)\right|^{2}dv\right|^{2}\right]du\\
 & \leq Cn\int_{\tau_{n}(s)}^{s}\left(u-\tau_{n}\left(s\right)\right)\int_{\tau_{n}\left(s\right)}^{u}\mathbb{E}^{Q}\left[\left|\nabla\left(Ah_{i}\right)^{T}\sigma\left(X\left(v\right)\right)\right|^{4}\right]dvdu\\
 & \leq C\left(T,h,b,\sigma\right)n\int_{\tau_{n}(s)}^{s}\left(u-\tau_{n}\left(s\right)\right)^{2}du\\
 & \leq C\left(T,h,b,\sigma\right)n^{-2}
\end{align*}
For the term $B_{2}$ one can use similar reasonings as for $B_{1}$
to obtain that 
\[
B_{2}\leq C\left(T,h,b,\sigma\right)n^{-4},
\]
and, hence, 
\[
\mathbb{E}^{Q}\left[\left|\left(\int_{\tau_{n}(s)}^{s}n\Delta_{n}^{i}\left(u\right)du\right)\right|^{4}\right]^{1/2}\leq Cn^{-1}
\]
On the other hand, using Hölder's innequality, the bounds for $L^{n}$
in Lemma 3.9 in \citep{CriO-L20}, that $\varphi$ has at most polynomial
growth and $\left(\ref{momest1}\right)$ we get that
\[
\sup_{s\in\left[0,T\right]}\sup_{n\in\mathbb{N}}\mathbb{E}^{Q}\left[\left|L^{n}\left(s\right)\varphi\left(X\left(s\right)\right)\right|^{4}\right]<\infty.
\]
Combining the previous estimates we can conclude that $\Phi_{t}^{n,i}\left(\varphi\right)=\mathcal{O}\left(n^{-1/2}\right).$

\end{proof}

\begin{lemma}\label{lem: aux2}

For all $i=1,...,d_{Y}$ and $\varphi\in C_{P}^{3}$, we have that
\begin{align*}
 & \lim_{N\rightarrow\infty}\frac{1}{N}\sum_{k=1}^{N}n\int_{0}^{t}L_{k}^{n}\left(s\right)\varphi\left(X_{k}\left(s\right)\right)\int_{\tau_{n}(s)}^{s}\nabla h_{i}^{T}\sigma\left(X_{k}\left(u\right)\right)dB_{k}\left(u\right)dY^{i}\left(s\right)\\
 & =\frac{1}{2\sqrt{3}}\sum_{i=1}^{d_{Y}}\int_{0}^{t}\rho_{\tau_{n}\left(s\right)}^{n}\left(\mathrm{tr}\left(\widetilde{O}_{\sigma,\varphi,h_{i}}\right)\right)dR^{n,i}\left(s\right)\\
 & +\frac{1}{2}\sum_{i=1}^{d_{Y}}\int_{0}^{t}\rho_{\tau_{n}\left(s\right)}^{n}\left(\mathrm{tr}\left(\widetilde{O}_{\sigma,\varphi,h_{i}}\right)\right)dY^{i}\left(s\right)+\Psi_{t}^{n,i}\left(\varphi\right),
\end{align*}
where $\Psi_{t}^{n,i}\left(\varphi\right)=\mathcal{O}\left(n^{-1/2}\right).$

\end{lemma}

\begin{proof}

We can write

\begin{align*}
 & n\int_{0}^{t}L_{k}^{n}\left(s\right)\varphi\left(X_{k}\left(s\right)\right)\int_{\tau_{n}(s)}^{s}\nabla h_{i}^{T}\sigma\left(X_{k}\left(u\right)\right)dB_{k}\left(u\right)dY^{i}\left(s\right)\\
 & =n\int_{0}^{t}\left\{ L_{k}^{n}\left(s\right)\varphi\left(X_{k}\left(s\right)\right)-L_{k}^{n}\left(\tau_{n}\left(s\right)\right)\varphi\left(X_{k}\left(\tau_{n}\left(s\right)\right)\right)\right\} \\
 & \qquad\times\int_{\tau_{n}(s)}^{s}\nabla h_{i}^{T}\sigma\left(X_{k}\left(u\right)\right)dB_{k}\left(u\right)dY^{i}\left(s\right)\\
 & +n\int_{0}^{t}L_{k}^{n}\left(\tau_{n}\left(s\right)\right)\varphi\left(X_{k}\left(\tau_{n}\left(s\right)\right)\right)\int_{\tau_{n}(s)}^{s}\nabla h_{i}^{T}\sigma\left(X_{k}\left(u\right)\right)dB_{k}\left(u\right)dY^{i}\left(s\right)\\
 & \triangleq C_{k,1}+C_{k,2}.
\end{align*}
Using similar reasonings as for the term $A_{1}$ in the proof of
Proposition \ref{Pr:Un}), one has that 
\[
\lim_{N\rightarrow\infty}\frac{1}{N}\sum_{k=1}^{N}C_{k,2}=0.
\]
Using integration by parts we obtain 
\begin{align*}
 & L_{k}^{n}\left(s\right)\varphi\left(X_{k}\left(s\right)\right)-L_{k}^{n}\left(\tau_{n}\left(s\right)\right)\varphi\left(X_{k}\left(\tau_{n}\left(s\right)\right)\right)\\
 & =\int_{\tau_{n}\left(s\right)}^{s}L_{k}^{n}\left(u\right)\varphi h^{T}\left(X_{k}\left(\tau_{n}\left(s\right)\right)\right)dY\left(u\right)\\
 & \quad+\int_{\tau_{n}\left(s\right)}^{s}L_{k}^{n}\left(u\right)A\varphi\left(X_{k}\left(u\right)\right)du+\int_{\tau_{n}\left(s\right)}^{s}L_{k}^{n}\left(u\right)\nabla\varphi^{T}\sigma\left(X_{k}\left(u\right)\right)dB_{k}\left(u\right)\\
 & \triangleq D_{k,1}\text{\ensuremath{\left(s\right)}}+D_{k,2}\left(s\right)+D_{k,3}\left(s\right),
\end{align*}
and using Itô's formula 
\begin{align*}
\int_{\tau_{n}(s)}^{s}\nabla h_{i}^{T}\sigma\left(X_{k}\left(u\right)\right)dB_{k}\left(u\right) & =\nabla h_{i}^{T}\sigma\left(X_{k}\left(\tau_{n}\left(s\right)\right)\right)\left(B_{k}\left(s\right)-B_{k}\left(\tau_{n}\left(s\right)\right)\right)\\
 & \quad+\int_{\tau_{n}(s)}^{s}\int_{\tau_{n}(s)}^{u}\nabla\left(\nabla h_{i}^{T}\sigma\right)^{T}\sigma\left(X_{k}\left(v\right)\right)dB_{k}\left(v\right)dB_{k}\left(u\right)\\
 & \quad+\int_{\tau_{n}(s)}^{s}\int_{\tau_{n}(s)}^{u}A\left(\nabla h_{i}^{T}\sigma\right)^{T}\left(X_{k}\left(v\right)\right)dvdB_{k}\left(u\right)\\
 & \triangleq E_{k,1}\text{\ensuremath{\left(s\right)}}+E_{k,2}\left(s\right)+E_{k,3}\left(s\right).
\end{align*}
Hence, the term $C_{k,1}$ can be written as the sum of nine terms
\[
C_{k,1}=\sum_{l,m=1}^{3}n\int_{0}^{t}D_{k,l}\text{\ensuremath{\left(s\right)}}E_{k,m}\text{\ensuremath{\left(s\right)}}dY^{i}\left(s\right).
\]
The terms containing as a factor $D_{k,2}\left(s\right),E_{k,2}\left(s\right)$,
and $E_{k,3}\left(s\right)$, after averaging over $1\leq k\leq N$
and taking limit when $N$ tends to infinity, yield processes which
are at least of order $\mathcal{O}\left(n^{-1/2}\right)$. With similar
reasonings as in Lemma \ref{lem: aux1} we can identify these processes
as:

\begin{align*}
\Psi_{t}^{n,i,1}\left(\varphi\right) & \triangleq n\int_{0}^{t}\mathbb{E}^{Q}\left[\int_{\tau_{n}\left(s\right)}^{s}L^{n}\left(u\right)A\varphi\left(X\left(u\right)\right)du\right.\\
 & \qquad\times\left.\left.\nabla h_{i}^{T}\sigma\left(X\left(\tau_{n}\left(s\right)\right)\right)\left(B\left(s\right)-B\left(\tau_{n}\left(s\right)\right)\right)\right|\mathcal{F}_{s}^{Y}\right]dY^{i}\left(s\right),\\
\Psi_{t}^{n,i,2}\left(\varphi\right) & \triangleq n\int_{0}^{t}\mathbb{E}^{Q}\left[\int_{\tau_{n}\left(s\right)}^{s}L^{n}\left(u\right)A\varphi\left(X\left(u\right)\right)du\right.\\
 & \qquad\times\left.\left.\int_{\tau_{n}(s)}^{s}\int_{\tau_{n}(s)}^{u}\nabla\left(\nabla h_{i}^{T}\sigma\right)^{T}\sigma\left(X\left(v\right)\right)dB\left(v\right)dB\left(u\right)\right|\mathcal{F}_{s}^{Y}\right]dY^{i}\left(s\right),\\
\Psi_{t}^{n,i,3}\left(\varphi\right) & \triangleq n\int_{0}^{t}\mathbb{E}^{Q}\left[\int_{\tau_{n}\left(s\right)}^{s}L^{n}\left(u\right)A\varphi\left(X\left(u\right)\right)du\right.\\
 & \qquad\times\left.\left.\int_{\tau_{n}(s)}^{s}\int_{\tau_{n}(s)}^{u}A\left(\nabla h_{i}^{T}\sigma\right)^{T}\left(X\left(v\right)\right)dvdB\left(u\right)\right|\mathcal{F}_{s}^{Y}\right]dY^{i}\left(s\right),\\
\Psi_{t}^{n,i,4}\left(\varphi\right) & \triangleq n\int_{0}^{t}\mathbb{E}^{Q}\left[\int_{\tau_{n}\left(s\right)}^{s}L^{n}\left(u\right)\varphi h^{T}\left(X\left(\tau_{n}\left(s\right)\right)\right)dY\left(u\right)\right.\\
 & \qquad\times\left.\left.\int_{\tau_{n}(s)}^{s}\int_{\tau_{n}(s)}^{u}\nabla\left(\nabla h_{i}^{T}\sigma\right)^{T}\sigma\left(X\left(v\right)\right)dB\left(v\right)dB\left(u\right)\right|\mathcal{F}_{s}^{Y}\right]dY^{i}\left(s\right),\\
\Psi_{t}^{n,i,5}\left(\varphi\right) & \triangleq n\int_{0}^{t}\mathbb{E}^{Q}\left[\int_{\tau_{n}\left(s\right)}^{s}L^{n}\left(u\right)\nabla\varphi^{T}\sigma\left(X\left(u\right)\right)dB\left(u\right)\right.\\
 & \qquad\times\left.\left.\int_{\tau_{n}(s)}^{s}\int_{\tau_{n}(s)}^{u}\nabla\left(\nabla h_{i}^{T}\sigma\right)^{T}\sigma\left(X\left(v\right)\right)dB\left(v\right)dB\left(u\right)\right|\mathcal{F}_{s}^{Y}\right]dY^{i}\left(s\right),\\
\Psi_{t}^{n,i,6}\left(\varphi\right) & \triangleq n\int_{0}^{t}\mathbb{E}^{Q}\left[\int_{\tau_{n}\left(s\right)}^{s}L^{n}\left(u\right)\varphi h^{T}\left(X\left(\tau_{n}\left(s\right)\right)\right)dY\left(u\right)\right.\\
 & \qquad\times\left.\left.\int_{\tau_{n}(s)}^{s}\int_{\tau_{n}(s)}^{u}A\left(\nabla h_{i}^{T}\sigma\right)^{T}\left(X\left(v\right)\right)dvdB\left(u\right)\right|\mathcal{F}_{s}^{Y}\right]dY^{i}\left(s\right),\\
\Psi_{t}^{n,i,7}\left(\varphi\right) & \triangleq n\int_{0}^{t}\mathbb{E}^{Q}\left[\int_{\tau_{n}\left(s\right)}^{s}L^{n}\left(u\right)\nabla\varphi^{T}\sigma\left(X\left(u\right)\right)dB\left(u\right)\right.\\
 & \qquad\times\left.\left.\int_{\tau_{n}(s)}^{s}\int_{\tau_{n}(s)}^{u}A\left(\nabla h_{i}^{T}\sigma\right)^{T}\left(X\left(v\right)\right)dvdB\left(u\right)\right|\mathcal{F}_{s}^{Y}\right]dY^{i}\left(s\right),
\end{align*}
There are two terms left:
\begin{align*}
F_{k,1}^{n,i}\left(t,\varphi\right) & \triangleq n\int_{0}^{t}\left(\int_{\tau_{n}\left(s\right)}^{s}L_{k}^{n}\left(u\right)\varphi\left(X_{k}\left(u\right)\right)h^{T}\left(X_{k}\left(\tau_{n}\left(s\right)\right)\right)dY\left(u\right)\right)\\
 & \qquad\times\nabla h_{i}^{T}\sigma\left(X_{k}\left(\tau_{n}\left(s\right)\right)\right)\left(B_{k}\left(s\right)-B_{k}\left(\tau_{n}\left(s\right)\right)\right)dY^{i}\left(s\right),
\end{align*}
and
\begin{align*}
F_{k,2}^{n,i}\left(t,\varphi\right) & \triangleq n\int_{0}^{t}\left(\int_{\tau_{n}\left(s\right)}^{s}L_{k}^{n}\left(u\right)\nabla\varphi^{T}\sigma\left(X_{k}\left(u\right)\right)dB_{k}\left(u\right)\right)\\
 & \qquad\times\nabla h_{i}^{T}\sigma\left(X_{k}\left(\tau_{n}\left(s\right)\right)\right)\left(B_{k}\left(s\right)-B_{k}\left(\tau_{n}\left(s\right)\right)\right)dY^{i}\left(s\right).
\end{align*}

For the term $F_{k,1}^{n,i}\left(t,\varphi\right)$, using integration
by parts with $L_{k}^{n}\left(u\right)\varphi\left(X_{k}\left(u\right)\right)$,
we can write 
\begin{align*}
F_{k,1}^{n,i}\left(t,\varphi\right) & =n\int_{0}^{t}L_{k}^{n}\left(\tau_{n}\left(s\right)\right)\varphi h^{T}\left(X_{k}\left(\tau_{n}\left(s\right)\right)\right)\left(Y\left(s\right)-Y\left(\tau_{n}\left(s\right)\right)\right)\\
 & \qquad\times\nabla h_{i}^{T}\sigma\left(X_{k}\left(\tau_{n}\left(s\right)\right)\right)\left(B_{k}\left(s\right)-B_{k}\left(\tau_{n}\left(s\right)\right)\right)dY^{i}\left(s\right)\\
 & +\Psi_{t}^{k,n,i,8}\left(\varphi\right)+\Psi_{t}^{k,n,i,9}\left(\varphi\right)+\Psi_{t}^{k,n,i,10}\left(\varphi\right)
\end{align*}
We get, using similar reasonings as for the term $C_{k,2}$, that
\[
\lim_{N\rightarrow\infty}\frac{1}{N}\sum_{k=1}^{N}n\int_{0}^{t}L_{k}^{n}\left(\tau_{n}\left(s\right)\right)\varphi h^{T}\left(X_{k}\left(\tau_{n}\left(s\right)\right)\right)\left(Y\left(s\right)-Y\left(\tau_{n}\left(s\right)\right)\right)=0,
\]
and the following terms
\begin{align*}
\Psi_{t}^{n,i,8}\left(\varphi\right) & \triangleq\lim_{N\rightarrow\infty}\frac{1}{N}\sum_{k=1}^{N}\Psi_{t}^{k,n,i,8}\left(\varphi\right)\\
 & =n\int_{0}^{t}\mathbb{E}^{Q}\left[\int_{\tau_{n}\left(s\right)}^{s}\int_{\tau_{n}\left(s\right)}^{u}L^{n}\left(r\right)\nabla\varphi^{T}\sigma\left(X\left(r\right)\right)dB\left(r\right)h^{T}\left(X\left(\tau_{n}\left(s\right)\right)\right)dY\left(u\right)\right.\\
 & \qquad\times\left.\left.\nabla h_{i}^{T}\sigma\left(X\left(\tau_{n}\left(s\right)\right)\right)\left(B\left(s\right)-B\left(\tau_{n}\left(s\right)\right)\right)\right|\mathcal{F}_{s}^{Y}\right]dY^{i}\left(s\right)\\
\Psi_{t}^{n,i,9}\left(\varphi\right) & \triangleq\lim_{N\rightarrow\infty}\frac{1}{N}\sum_{k=1}^{N}\Psi_{t}^{k,n,i,9}\left(\varphi\right)\\
 & =n\int_{0}^{t}\mathbb{E}^{Q}\left[\int_{\tau_{n}\left(s\right)}^{s}\int_{\tau_{n}\left(s\right)}^{u}L^{n}\left(r\right)A\varphi\left(X\left(r\right)\right)drh^{T}\left(X\left(\tau_{n}\left(s\right)\right)\right)dY\left(u\right)\right.\\
 & \qquad\times\left.\left.\nabla h_{i}^{T}\sigma\left(X\left(\tau_{n}\left(s\right)\right)\right)\left(B\left(s\right)-B\left(\tau_{n}\left(s\right)\right)\right)\right|\mathcal{F}_{s}^{Y}\right]dY^{i}\left(s\right)\\
\Psi_{t}^{n,i,10}\left(\varphi\right) & \triangleq\lim_{N\rightarrow\infty}\frac{1}{N}\sum_{k=1}^{N}\Psi_{t}^{k,n,i,10}\left(\varphi\right)\\
 & =n\int_{0}^{t}\mathbb{E}^{Q}\left[\int_{\tau_{n}\left(s\right)}^{s}\int_{\tau_{n}\left(s\right)}^{u}L^{n}\left(r\right)\varphi\left(X\left(r\right)\right)h^{T}\left(X\left(\tau_{n}\left(s\right)\right)\right)dY\left(r\right)h^{T}\left(X\left(\tau_{n}\left(s\right)\right)\right)dY\left(u\right)\right.\\
 & \qquad\times\left.\left.\nabla h_{i}^{T}\sigma\left(X\left(\tau_{n}\left(s\right)\right)\right)\left(B\left(s\right)-B\left(\tau_{n}\left(s\right)\right)\right)\right|\mathcal{F}_{s}^{Y}\right]dY^{i}\left(s\right),
\end{align*}
which are at least of order $\mathcal{O}\left(n^{-1/2}\right)$.

For the term $F_{k,2}^{n,i}\left(t,\varphi\right)$, using integration
by parts with $L_{k}^{n}\left(u\right)\nabla\varphi^{T}\sigma\left(X_{k}\left(u\right)\right)$,
we can write 
\begin{align}
F_{k,2}^{n,i}\left(t,\varphi\right) & =n\int_{0}^{t}L_{k}^{n}\left(\tau_{n}\left(s\right)\right)\nabla\varphi^{T}\sigma\left(X_{k}\left(\tau_{n}\left(s\right)\right)\right)\left(B_{k}\left(s\right)-B_{k}\left(\tau_{n}\left(s\right)\right)\right)\nonumber \\
 & \qquad\times\nabla h_{i}^{T}\sigma\left(X_{k}\left(\tau_{n}\left(s\right)\right)\right)\left(B_{k}\left(s\right)-B_{k}\left(\tau_{n}\left(s\right)\right)\right)dY^{i}\left(s\right)\nonumber \\
 & +\Psi_{t}^{k,n,i,11}\left(\varphi\right)+\Psi_{t}^{k,n,i,12}\left(\varphi\right)+\Psi_{t}^{k,n,i,13}\left(\varphi\right).\label{eq:Fnik2}
\end{align}
Denote the first term on the right hand side of the previous equation
by $\widehat{F}_{k,2}^{n,i}\left(t,\varphi\right)$, and consider
the following matrices

\[
O_{\sigma,\varphi,h_{i}}\left(x\right)\triangleq\sigma^{T}\left(x\right)\nabla\varphi\left(x\right)\nabla h_{i}^{T}\left(x\right)\sigma\left(x\right),
\]
and 
\[
\widetilde{O}_{\sigma,\varphi,h_{i}}\left(x\right)\triangleq\frac{1}{2}\left(O_{\sigma,\varphi,h_{i}}\left(x\right)+O_{\sigma,\varphi,h_{i}}^{T}\left(x\right)\right).
\]
Then, $\widehat{F}_{k,2}^{n,i}\left(t,\varphi\right)$ can be further
expanded using again integration by parts 
\begin{align*}
\widehat{F}_{k,2}^{n,i}\left(t,\varphi\right) & =n\int_{0}^{t}L_{k}^{n}\left(\tau_{n}\left(s\right)\right)\nabla\varphi^{T}\sigma\left(X_{k}\left(\tau_{n}\left(s\right)\right)\right)\left(B_{k}\left(s\right)-B_{k}\left(\tau_{n}\left(s\right)\right)\right)\\
 & \qquad\qquad\times\nabla h_{i}^{T}\sigma\left(X_{k}\left(\tau_{n}\left(s\right)\right)\right)\left(B_{k}\left(s\right)-B_{k}\left(\tau_{n}\left(s\right)\right)\right)dY^{i}\left(s\right).\\
 & =n\int_{0}^{t}L_{k}^{n}\left(\tau_{n}\left(s\right)\right)\left(B_{k}\left(s\right)-B_{k}\left(\tau_{n}\left(s\right)\right)\right)^{T}O_{\sigma,\varphi,h_{i}}\left(X_{k}\left(\tau_{n}\left(s\right)\right)\right)\left(B_{k}\left(s\right)-B_{k}\left(\tau_{n}\left(s\right)\right)\right)dY^{i}\left(s\right)\\
 & =n\int_{0}^{t}L_{k}^{n}\left(\tau_{n}\left(s\right)\right)\int_{\tau_{n}\left(s\right)}^{s}\left(B_{k}\left(u\right)-B_{k}\left(\tau_{n}\left(s\right)\right)\right)^{T}\left(O_{\sigma,\varphi,h_{i}}+O_{\sigma,\varphi,h_{i}}^{T}\right)\left(X_{k}\left(\tau_{n}\left(s\right)\right)\right)dB_{k}\left(u\right)dY^{i}\left(s\right)\\
 & \quad+\frac{n}{2}\int_{0}^{t}L_{k}^{n}\left(\tau_{n}\left(s\right)\right)\mathrm{tr}\left(\left(O_{\sigma,\varphi,h_{i}}+O_{\sigma,\varphi,h_{i}}^{T}\right)\left(X_{k}\left(\tau_{n}\left(s\right)\right)\right)\right)\left(s-\tau_{n}\left(s\right)\right)dY^{i}\left(s\right)\\
 & =2n\int_{0}^{t}L_{k}^{n}\left(\tau_{n}\left(s\right)\right)\int_{\tau_{n}\left(s\right)}^{s}\left(B_{k}\left(u\right)-B_{k}\left(\tau_{n}\left(s\right)\right)\right)^{T}\widetilde{O}_{\sigma,\varphi,h_{i}}\left(X_{k}\left(\tau_{n}\left(s\right)\right)\right)dB_{k}\left(u\right)dY^{i}\left(s\right)\\
 & \quad+n\int_{0}^{t}L_{k}^{n}\left(\tau_{n}\left(s\right)\right)\mathrm{tr}\left(\widetilde{O}_{\sigma,\varphi,h_{i}}\left(X_{k}\left(\tau_{n}\left(s\right)\right)\right)\right)\left(s-\tau_{n}\left(s\right)\right)dY^{i}\left(s\right)
\end{align*}
On the one hand, using similar reasonings as for the term $C_{k,2},$
we have 
\[
\lim_{N\rightarrow\infty}\frac{1}{N}\sum_{k=1}^{N}2n\int_{0}^{t}L_{k}^{n}\left(\tau_{n}\left(s\right)\right)\int_{\tau_{n}\left(s\right)}^{s}\left(B_{k}\left(u\right)-B_{k}\left(\tau_{n}\left(s\right)\right)\right)^{T}\widetilde{O}_{\sigma,\varphi,h_{i}}\left(X_{k}\left(\tau_{n}\left(s\right)\right)\right)dB_{k}\left(u\right)dY^{i}\left(s\right)=0.
\]
On the other hand, recalling the expression for $n\left(s-\tau_{n}\left(s\right)\right)dY^{i}\left(s\right)$,
we get 
\begin{align*}
 & n\int_{0}^{t}L_{k}^{n}\left(\tau_{n}\left(s\right)\right)\mathrm{tr}\left(\widetilde{O}_{\sigma,\varphi,h_{i}}\left(X_{k}\left(\tau_{n}\left(s\right)\right)\right)\right)\left(s-\tau_{n}\left(s\right)\right)dY^{i}\left(s\right)\\
 & =\frac{1}{2\sqrt{3}}\int_{0}^{t}L_{k}^{n}\left(\tau_{n}\left(s\right)\right)\mathrm{tr}\left(\widetilde{O}_{\sigma,\varphi,h_{i}}\left(X_{k}\left(\tau_{n}\left(s\right)\right)\right)\right)dR^{n,i}\left(s\right)\\
 & \quad+\frac{1}{2}\int_{0}^{t}L_{k}^{n}\left(\tau_{n}\left(s\right)\right)\mathrm{tr}\left(\widetilde{O}_{\sigma,\varphi,h_{i}}\left(X_{k}\left(\tau_{n}\left(s\right)\right)\right)\right)dY^{i}\left(s\right),
\end{align*}
and 
\begin{align*}
 & \lim_{N\rightarrow\infty}\frac{1}{N}\sum_{k=1}^{N}\frac{1}{2\sqrt{3}}\int_{0}^{t}L_{k}^{n}\left(\tau_{n}\left(s\right)\right)\mathrm{tr}\left(\widetilde{O}_{\sigma,\varphi,h_{i}}\left(X_{k}\left(\tau_{n}\left(s\right)\right)\right)\right)dR^{n,i}\left(s\right)\\
 & =\frac{1}{2\sqrt{3}}\int_{0}^{t}\rho_{\tau_{n}\left(s\right)}^{n}\left(\mathrm{tr}\left(\widetilde{O}_{\sigma,\varphi,h_{i}}\right)\right)dR^{n,i}\left(s\right),\\
 & \lim_{N\rightarrow\infty}\frac{1}{N}\sum_{k=1}^{N}\frac{1}{2}\int_{0}^{t}L_{k}^{n}\left(\tau_{n}\left(s\right)\right)\mathrm{tr}\left(\widetilde{O}_{\sigma,\varphi,h_{i}}\left(X_{k}\left(\tau_{n}\left(s\right)\right)\right)\right)dY^{i}\left(s\right)\\
 & =\frac{1}{2}\int_{0}^{t}\rho_{\tau_{n}\left(s\right)}^{n}\left(\mathrm{tr}\left(\widetilde{O}_{\sigma,\varphi,h_{i}}\right)\right)dY^{i}\left(s\right)
\end{align*}
With similar reasonings as in Lemma \ref{lem: aux1}, we can identify
the processes in $\left(\ref{eq:Fnik2}\right)$ after the averaging
and limiting procedure. That is
\begin{align*}
\Psi_{t}^{n,i,11}\left(\varphi\right) & \triangleq\lim_{N\rightarrow\infty}\frac{1}{N}\sum_{k=1}^{N}\Psi_{t}^{k,n,i,11}\left(\varphi\right)\\
 & =n\int_{0}^{t}\mathbb{E}^{Q}\left[\int_{\tau_{n}\left(s\right)}^{s}\int_{\tau_{n}\left(s\right)}^{u}L^{n}\left(r\right)\nabla\left(\nabla\varphi^{T}\sigma\right)^{T}\sigma\left(X\left(r\right)\right)dB\left(r\right)dB\left(u\right)\right.\\
 & \qquad\times\left.\left.\nabla h_{i}^{T}\sigma\left(X\left(\tau_{n}\left(s\right)\right)\right)\left(B\left(s\right)-B\left(\tau_{n}\left(s\right)\right)\right)\right|\mathcal{F}_{s}^{Y}\right]dY^{i}\left(s\right)\\
\Psi_{t}^{n,i,12}\left(\varphi\right) & \triangleq\lim_{N\rightarrow\infty}\frac{1}{N}\sum_{k=1}^{N}\Psi_{t}^{k,n,i,12}\left(\varphi\right)\\
 & =n\int_{0}^{t}\mathbb{E}^{Q}\left[\int_{\tau_{n}\left(s\right)}^{s}\int_{\tau_{n}\left(s\right)}^{u}L^{n}\left(r\right)A\left(\nabla\varphi^{T}\sigma\right)^{T}\sigma\left(X\left(r\right)\right)dB\left(r\right)dB\left(u\right)\right.\\
 & \qquad\times\left.\left.\nabla h_{i}^{T}\sigma\left(X\left(\tau_{n}\left(s\right)\right)\right)\left(B\left(s\right)-B\left(\tau_{n}\left(s\right)\right)\right)\right|\mathcal{F}_{s}^{Y}\right]dY^{i}\left(s\right)\\
\Psi_{t}^{n,i,13}\left(\varphi\right) & \triangleq\lim_{N\rightarrow\infty}\frac{1}{N}\sum_{k=1}^{N}\Psi_{t}^{k,n,i,13}\left(\varphi\right)\\
 & =n\int_{0}^{t}\mathbb{E}^{Q}\left[\int_{\tau_{n}\left(s\right)}^{s}\int_{\tau_{n}\left(s\right)}^{u}L^{n}\left(r\right)\nabla\varphi^{T}\sigma\left(X\left(r\right)\right)h^{T}\left(X\left(\tau_{n}\left(s\right)\right)\right)dY\left(r\right)dB\left(u\right)\right.\\
 & \qquad\times\left.\left.\nabla h_{i}^{T}\sigma\left(X\left(\tau_{n}\left(s\right)\right)\right)\left(B\left(s\right)-B\left(\tau_{n}\left(s\right)\right)\right)\right|\mathcal{F}_{s}^{Y}\right]dY^{i}\left(s\right),
\end{align*}
 which are at least of order $\mathcal{O}\left(n^{-1/2}\right)$.
Finally, $\Psi_{t}^{n,i}\left(\varphi\right)=\sum_{j=1}^{13}\Psi_{t}^{n,i,j}\left(\varphi\right)$.

\end{proof}


\global\long\def\theequation{A.\arabic{equation}}%

\appendix

\section{Appendix}

\subsection{Limits for particle representations}

\label{limpr} As we will see, particle representations are useful
in deriving approximations and computing limits. In the following
lemma from \citep{KK10}, $N_{n}$ may be finite or infinite. If $N_{n}=\infty$,
then by $\frac{1}{N_{n}}\sum_{k=1}^{N_{n}}z_{k}$, we mean $\lim_{m\rightarrow\infty}\frac{1}{k}\sum_{k=1}^{m}z_{k}$.

\begin{lemma}\label{exlimproc} Let $Z^{n}=(Z_{1}^{n},\ldots,Z_{N_{n}}^{n})$
be exchangeable families of $D_{E}[0,\infty)$-valued random variables
such that $N_{n}\Rightarrow\infty$ and $Z^{n}\Rightarrow Z$ in $D_{E}[0,\infty)$$^{\infty}$.
Define

\quad{}$\Xi^{n}=\frac{1}{N_{n}}\sum_{k=1}^{N_{n}}\delta_{Z_{k}^{n}}\in{\cal P}(D_{E}[0,\infty))$

\quad{}$\Xi=\lim_{m\rightarrow\infty}\frac{1}{m}\sum_{k=}^{m}\delta_{Z_{k}}$

\quad{}$V^{n}(t)=\frac{1}{N_{n}}\sum_{k=1}^{N_{n}}\delta_{Z_{k}^{n}(t)}\in{\cal P}(E)$

\quad{}$V(t)=\lim_{m\rightarrow\infty}\frac{1}{m}\sum_{k=1}^{m}\delta_{Z_{k}(t)}$

Then
\begin{itemize}
\item[a)] For $t_{1},\ldots,t_{l}\notin\{t:\mathbb{E}\left[\Xi\left\{ x:x(t)\neq x(t-)\right\} \right]>0\}$
\[
\left(\Xi_{n},V^{n}(t_{1}),\ldots,V^{n}(t_{l})\right)\Rightarrow\left(\Xi,V(t_{1}),\ldots,V(t_{l})\right).
\]
\item[b)] If $Z^{n}\Rightarrow Z$ in $D_{E^{\infty}}[0,\infty)$, then $V^{n}\Rightarrow V$
in $D_{{\cal P}(E)}[0,\infty)$. If $Z^{n}\rightarrow Z$ in probability
in $D_{E^{\infty}}[0,\infty)$, then $V^{n}\rightarrow V$ in $D_{{\cal P}(E)}[0,\infty)$
in probability. 
\end{itemize}
\end{lemma}

\begin{remark}\label{cEinf} If the $Z_{k}^{n}$ are in $C_{E}[0,\infty)$,
then every $D$ can be replaced by $C$. In particular, in this case,
$D_{E}[0,\infty)^{\infty}=C_{E^{\infty}}[0,\infty)$. \end{remark}

\begin{remark}\label{intcnv} If $z^{n}$ converges in $D_{{\cal P}(E)}[0,\infty)$
to $z\in D_{{\cal P}(E)}[0,\infty)$, then for any bounded continuous
$f:E\rightarrow\mathbb{R}$, $x^{n}$ given by $x^{n}(t)=\int_{E}f(u)z^{n}(t,du)$
converges to $x$ given by $x(t)=\int_{E}f(u)z(t,du)$ in $D_{\mathbb{R}}[0,\infty)$.
We need similar results for unbounded $f$. \end{remark}

\begin{lemma}\label{wkext} Let $E$ be locally compact and $C_{0}(E)$
be the space of continuous functions vanishing at infinity. Let $\left\{ \mu_{n}\right\} \subset{\cal P}(E)$
and $\mu_{n}\Rightarrow\mu$. Suppose $\varphi\in C(E)$, $\varphi>0$,
and $\sup_{n}\int_{E}\varphi d\mu_{n}<\infty$. Then if $f\in C(E)$
and $\frac{f}{\varphi}\in C_{0}(E)$, $\lim_{n\rightarrow\infty}\int_{E}fd\mu_{n}=\int_{E}fd\mu$.
\end{lemma}

\begin{proof} Let $E^{\infty}$ denote the one point compactification
and $d\nu_{n}=\varphi d\mu_{n}$. Then since $\sup_{n}\nu_{n}(E^{\infty})<\infty$,
$\{\nu_{n}\}$ is relatively compact in ${\cal M}_{f}(E^{\infty})$,
the finite measures on $E^{\infty}$, and every limit point is of
the form $\varphi\mu+p\delta_{\infty}$. Consequently, if $\frac{f}{\varphi}\in C_{0}(E)$,
\[
\lim_{n\rightarrow\infty}\int_{E}fd\mu_{n}=\lim_{n\rightarrow\infty}\int_{E_{\infty}}\frac{f}{\varphi}d\nu_{n}=\int_{E}fd\mu.
\]
\end{proof}\ 

\begin{lemma}\label{imgcnv} Suppose $z^{n}$ converges in $D_{{\cal P}(E)}[0,\infty)$
to $z\in D_{{\cal P}(E)}[0,\infty)$ and for each $T=1,2,\cdots$,
there exists $\varphi_{T}\in C(E)$, $\varphi_{T}>0$, such that $\sup_{n}\sup_{t\leq T}\int_{E}\varphi_{T}(u)z^{n}(t,du)<\infty$.
If $f\in C(E)$ and for each $T$, $\frac{f}{\varphi_{T}}\in C_{0}(E)$,
then $y_{n}(t)=\int_{E}f(u)z^{n}(t,du)$ converges to $y$ given by
$y(t)=\int_{E}f(u)z(t,du)$ in $D_{\mathbb{R}}[0,\infty)$. \end{lemma}

\begin{proof} Suppose $t_{n}\rightarrow t$. By Lemma \ref{wkext},
if $z^{n}(t_{n})\rightarrow z(t)$, then $y_{n}(t_{n})\rightarrow y(t)$
and if $z^{n}(t_{n})\rightarrow z(t-)$, $y_{n}(t_{n})\rightarrow y(t-)$.
Since $\{z^{n}\}$ converges in $D_{{\cal P}(E)}[0,\infty)$, by Lemma
3.6.5 of \citep{EK86}, $y_{n}\rightarrow y$ in $D_{E}[0,\infty)$.
\end{proof}

\begin{lemma}\label{exchconv} Let $\{\xi_{k}\}\subset C_{\mathbb{R}}[0,T]$
be exchangeable with 
\begin{equation}
\mathbb{E}\left[\sup_{0\leq t\leq T}\left|\xi_{k}(t)\right|\right]<\infty,\label{supnrm}
\end{equation}
and let ${\cal T}_{t}=\cap_{N}\sigma\left(\xi_{N}(s),\xi_{N+1}(s),\ldots:s\leq t\right)$.
Then setting $\hat{\xi}(t)=\mathbb{E}\left[\xi_{1}(t)|{\cal T}_{t}\right]$
\[
\lim_{N\rightarrow\infty}\sup_{0\leq t\leq T}\left|\frac{1}{N}\sum_{k=1}^{N}\xi_{k}(t)-\mathbb{E}\left[\xi_{1}(t)|{\cal T}_{t}\right]\right|.\quad a.s.
\]
\end{lemma}

\begin{proof} Let $U_{k}^{0}(t)=\xi_{k}([t])$ and 
\[
U_{k}^{j}(t)=\xi_{k}\left(\frac{[2^{j}t]}{2^{j}}\right)-\xi_{k}\left(\frac{[2^{j-1}t]}{2^{j-1}}\right).
\]
Then 
\[
\xi_{k}(t)=\sum_{j=0}^{\infty}U_{k}^{j}(t)
\]
and 
\[
\lim_{N\rightarrow\infty}\frac{1}{N}\sum_{k=1}^{N}\xi_{k}(t)=\sum_{j=0}^{\infty}\lim_{N\rightarrow\infty}\frac{1}{N}\sum_{k=1}^{N}U_{k}^{j}(t)\quad\mbox{{\rm for all }}t\quad a.s.
\]
Note that the convergence on the right involves countably many applications
of de Finetti's theorem for real exchangeable sequences $\left\{ \left\{ \xi_{k}\left(\frac{i}{2^{j}}\right)\right\} :i,j=0,1,\ldots\right\} $.
The interchange of the limit and the sum follows from (\ref{supnrm}).
Since we are assuming all $\sigma$-algebras are completed, 
\[
{\cal T}_{t}=\lor_{\frac{i}{2^{j}}\leq t}\cap_{N}\sigma\left(\xi_{N}\left(\frac{i}{2^{j}}\right),\xi_{N+1}\left(\frac{i}{2^{j}}\right),\ldots\right).
\]
\end{proof}

\subsection{Martingales and change of measure.}

\label{sectmcm}\ This section follows \citet*{Prot04}, Section
III.8. Let $\{{\cal F}_{t}\}$ be a filtration and assume that $P\vert_{{\cal F}_{t}}<<Q\vert_{{\cal F}_{t}}$,
for all $t\geq0$, and that $L(t)$ is the corresponding Radon-Nikodym
derivative. Then $L$ is an $\{{\cal F}_{t}\}$-martingale on $(\Omega,{\cal F},Q)$.

\begin{lemma}\label{cmmart} $Z$ is a $P$-local martingale if and
only if $LZ$ is a $Q$-local martingale. \end{lemma}

\begin{proof} Note that for a bounded stopping time $\tau$, $Z(\tau)$
is $P$-integrable if and only if $L(\tau)Z(\tau)$ is $Q$-integrable.
By Bayes formula, $\mathbb{E}^{P}[Z(t+h)-Z(t)|{\cal F}_{t}]=0$ if
and only if $\mathbb{E}^{Q}[L(t+h)(Z(t+h)-Z(t))|{\cal F}_{t}]=0$
which is equivalent to 
\[
\mathbb{E}^{Q}\left[L(t+h)Z(t+h)|{\cal F}_{t}\right]=\mathbb{E}^{Q}\left[L(t+h)Z(t)|{\cal F}_{t}\right]=L(t)Z(t).
\]
\end{proof}

\begin{theorem}\label{cmmart2} If $M$ is a $Q$-local martingale,
then 
\begin{equation}
Z(t)=M(t)-\int_{0}^{t}\frac{1}{L(s)}d[L,M]_{s}\label{zdef}
\end{equation}
is a $P$-local martingale. (Note that the integrand is $\frac{1}{L(s)}$,
not $\frac{1}{L(s-)}$.) \end{theorem}

\begin{proof} Note that $LM-[L,M]$ is a $Q$-local martingale. We
need to show that $LZ$ is a $Q$-local martingale. But letting $V$
denote the second term on the right of (\ref{zdef}), we have 
\[
L(t)Z(t)=L(t)M(t)-[L,M]_{t}-\int_{0}^{t}V(s-)dL(s),
\]
and both terms on the right are $Q$-local martingales.\hfill{}\end{proof}

\section*{Acknowledgement}
Part of this research was funded within the project STORM: Stochastics for Time-Space Risk Models, from the Research Council of Norway (RCN). Project number: 274410.

\bibliographystyle{plainnat}
\bibliography{filtrep}

\end{document}